\documentclass[a4paper,
               11pt,
               pdftex,
               normalheadings,
               headsepline,
               footsepline,
               onecolumn,
               headinclude,
               footinclude,
               DIV14,
               abstracton
               ]{scrartcl}

\usepackage
{
    graphicx,
    amssymb,
    amsmath,
    amsthm,
    xcolor,
    dsfont,
    algpseudocode,
    authblk,
}

\usepackage[bf]{caption}
\usepackage[colorlinks,
            pdffitwindow=false,
            plainpages=false,
            pdfpagelabels=true,
            pdfpagemode=UseOutlines,
            pdfpagelayout=SinglePage,
            bookmarks=false,
            colorlinks=true,
            hyperfootnotes=false,
            linkcolor=blue,
            citecolor=green!50!black,
            hyperindex,breaklinks]{hyperref}

\DeclareMathAlphabet{\mathpzc}{OT1}{pzc}{m}{it}

\newcommand{\subfiguretitle}[1]{{\scriptsize{#1}} \\[1mm] }
\newcommand{\R}{\mathbb{R}}
\newcommand{\C}{\mathbb{C}}

\providecommand{\abs}[1]{\left\lvert #1 \right\rvert}

\newcommand\xqed[1]{\leavevmode\unskip\penalty9999 \hbox{}\nobreak\hfill \quad\hbox{#1}}
\newcommand{\exampleSymbol}{\xqed{$\triangle$}}

\DeclareMathOperator{\trace}{trace}
\DeclareMathOperator{\rank}{rank}
\DeclareMathOperator{\Aut}{Aut}

\newtheorem{theorem}{Theorem}[section]

\newtheorem{lemma}[theorem]{Lemma}
\newtheorem{proposition}[theorem]{Proposition}
\newtheorem{definition}[theorem]{Definition}
\newtheorem{remark}[theorem]{Remark}
\newtheorem{example}[theorem]{Example}

\makeatletter
\renewcommand*\env@matrix[1][*\c@MaxMatrixCols c]{%
  \hskip -\arraycolsep
  \let\@ifnextchar\new@ifnextchar
  \array{#1}}
\makeatother

\begin{document}

\title{Sensing and Control in Symmetric Networks}
\author[*]{Michael Dellnitz}
\author[**]{Stefan Klus}
\affil[*]{\normalsize Department of Mathematics, University of Paderborn, D-33095 Paderborn, Germany}
\affil[**]{\normalsize Department of Mathematics and Computer Science, Freie Universit\"at Berlin, D-14195 Berlin, Germany}

\maketitle

\begin{abstract}
In engineering applications, one of the major challenges today is to develop reliable and robust control algorithms for complex networked systems. Controllability and observability of such systems play a crucial role in the design process. The underlying network structure may contain symmetries -- caused for example by the coupling of identical building blocks -- and these symmetries lead to repeated eigenvalues in a generic way. This complicates the design of controllers since repeated eigenvalues might decrease the controllability of the system. In this paper, we will analyze the relationship between the controllability and observability of complex networked systems and graph symmetries using results from representation theory. Furthermore, we will propose an algorithm to compute sparse input and output matrices based on projections onto the isotypic components. We will illustrate our results with the aid of two guiding examples, a network with $ D_4 $ symmetry and the Petersen graph.
\end{abstract}

\pagestyle{myheadings}
\thispagestyle{plain}

\begin{center}
This paper is dedicated to Marty Golubitsky on the occasion of his 70\textsuperscript{th} birthday.  
\end{center}

\section{Introduction}

In several real world applications, the crucial task is to design reliable control and sensing techniques for complex networked dynamical systems. Examples of such networks are, to name but a few, electric circuits, power grids, buildings, multi-agent systems, consensus problems, or social networks. These complex systems may exhibit symmetries in a variety of different ways due to the construction of the network or due to intrinsic characteristic properties of the system itself. It is well known that equivariance properties lead to the existence of multiple eigenvalues in a generic way. In bifurcation theory this has first been recognized by Marty Golubitsky and Ian Stewart (cf.~their ground-breaking book~\cite{GSS88}). But also in control theory this fact has already been realized, see for example~\cite{RM72}. Also in the recently published paper~\cite{WBSS15}, it is shown that symmetries in a network might decrease controllability and observability depending on the specific symmetry type. In~\cite{CM14}, the authors analyze the relationship between the controllability of a networked system and graph symmetries by utilizing so-called signed fractional automorphisms, a relaxation of the definition of an automorphism allowing not only permutation matrices but also matrices whose row and column sums are equal to one. The controllability and observability of networked linear systems has also been analyzed in \cite{FuHe13}. However, the related characterization is based on the extension of the notion of strict system equivalence to networks, and equivariance properties are not directly taken into account.

In this paper, we will exploit results from representation theory to determine whether a given system is controllable and observable and derive necessary and sufficient conditions. The idea of representation theory is to formulate abstract algebra problems as linear algebra problems. Each element of a group, for instance, can then be represented by an invertible matrix. Using this approach, it is possible to compute projections onto invariant subspaces and to decompose the problem into smaller subproblems, resulting in a block diagonal structure. Additionally, we will propose a new algorithm to compute sparse input and output matrices $ B $ and $ C $ for a given control problem. This is an important problem, in particular when only a limited number of control and observation input nodes is available~\cite{WBSS15}.

The paper is organized as follows: In Section~\ref{sec:Basic Concepts of Equivariance and Representations}, the basic concepts of equivariance and representation theory will be introduced. Section~\ref{sec:Control Theory} shows how results from representation theory can be used to determine controllability and observability of networked systems with inherent symmetries. The results will be illustrated with the aid of guiding examples. Section~\ref{sec:Conclusion} concludes with a brief summary and possible future work.

\section{Basic Concepts of Equivariance and Representations}
\label{sec:Basic Concepts of Equivariance and Representations}

In this section, we will review the basic concepts of equivariant dynamical systems theory which will be relevant for the considerations within this article. For background material and a description of the mathematical details, the reader is referred to the classical textbooks~\cite{GSS88, FS92}.

For simplicity, we initially restrict our attention to (uncontrolled) linear dynamical systems of the form
\begin{equation}
    \dot{x} = A x,
\end{equation}
where $ x \in \R^n $ and $ A \in \R^{n,n} $. Controllability and observability will be analyzed in the following section.

\subsection{Equivariance}

In the application we have in mind, the network structure of the underlying system induces certain structural properties of the matrix $ A $. In mathematical terms, this can be expressed by the following \emph{equivariance condition}: We assume that
\begin{equation}\label{eq:Aeq}
    \gamma A = A \gamma
\end{equation}
for all $ \gamma \in \Gamma $, where $ \Gamma \subset {\bf O}(n) $ is a finite group of orthogonal matrices.

\begin{example} \label{ex:network1}
The matrix
\begin{equation*}
    A =
    \begin{pmatrix}
        B   & C_1 & 0   & C_2 \\
        C_2 & B   & C_1 & 0   \\
        0   & C_2 & B   & C_1 \\
        C_1 & 0   & C_2 & B 
    \end{pmatrix},
\end{equation*}
with $ B, C_1, C_2 \in \R^{d,d} $, is equivariant with respect to the action of the cyclic group $ {\mathbb Z}_4 $ generated by the group element
\begin{equation*}
    R_1 = 
    \begin{pmatrix}
        0   & I_d & 0   & 0   \\
        0   & 0   & I_d & 0   \\
        0   & 0   & 0   & I_d \\
        I_d & 0   & 0   & 0 
    \end{pmatrix},
\end{equation*}
where here and in what follows $ I_d $ denotes the $ d $-dimensional identity matrix.

In fact, the matrix $ A $ reflects the structural properties of the network illustrated in Figure~\ref{fig:Z4_D4}(a). This network is invariant under rotations by 90, 180 and 270 degrees. Note that if the coupling structure to the nearest neighbors is identical, then the network is also invariant under reflections as shown in Figure~\ref{fig:Z4_D4}(b). In that case, $ C_1 = C_2 $ and the matrix $ A $ is equivariant with respect to the group $ D_4 $ generated by the group elements
\begin{equation*}
    R_1 =
    \begin{pmatrix}
        0   & I_d & 0   & 0   \\
        0   & 0   & I_d & 0   \\
        0   & 0   & 0   & I_d \\
        I_d & 0   & 0   & 0
    \end{pmatrix}
    \quad \text{and} \quad
    S_1 =
    \begin{pmatrix}
        0   & 0   & I_d & 0   \\
        0   & I_d & 0   & 0   \\
        I_d & 0   & 0   & 0   \\
        0   & 0   & 0   & I_d 
    \end{pmatrix}.
\end{equation*}
In general, the \emph{dihedral group} $ D_k $ consists of $ 2 k $ elements: $ k $ rotations $ R_0(=I_n), R_1, \ldots, R_{k-1} $ and $ k $ reflections $ S_1, \ldots, S_k $. \exampleSymbol

\begin{figure}[htb]
    \centering
    \begin{minipage}[t]{0.45\textwidth}
        \centering
        \subfiguretitle{(a)}
        \includegraphics[width=0.5\textwidth]{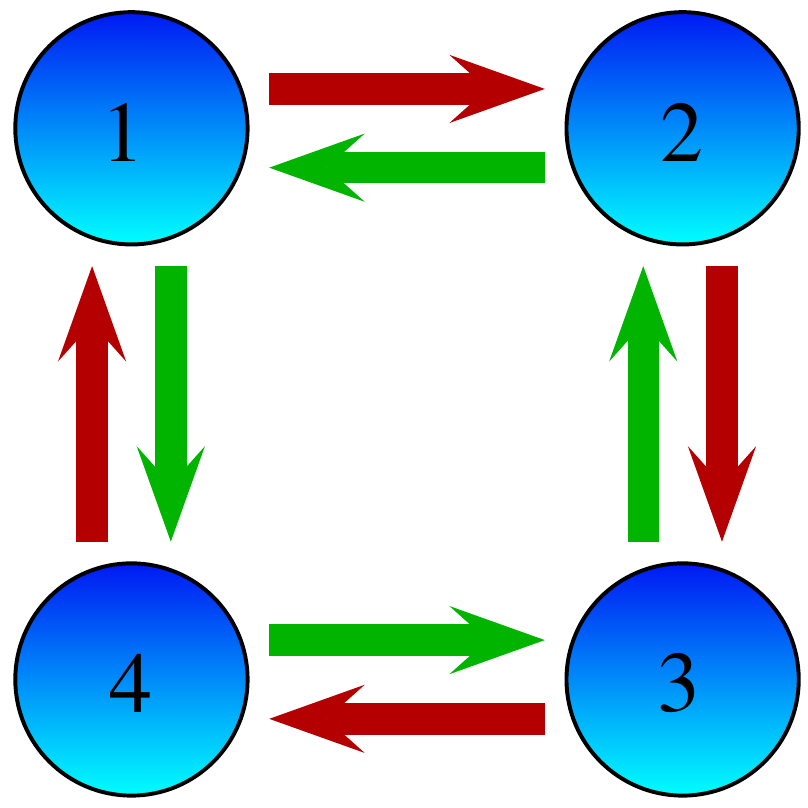} 
    \end{minipage}
    \begin{minipage}[t]{0.45\textwidth}
        \centering
        \subfiguretitle{(b)}
        \includegraphics[width=0.5\textwidth]{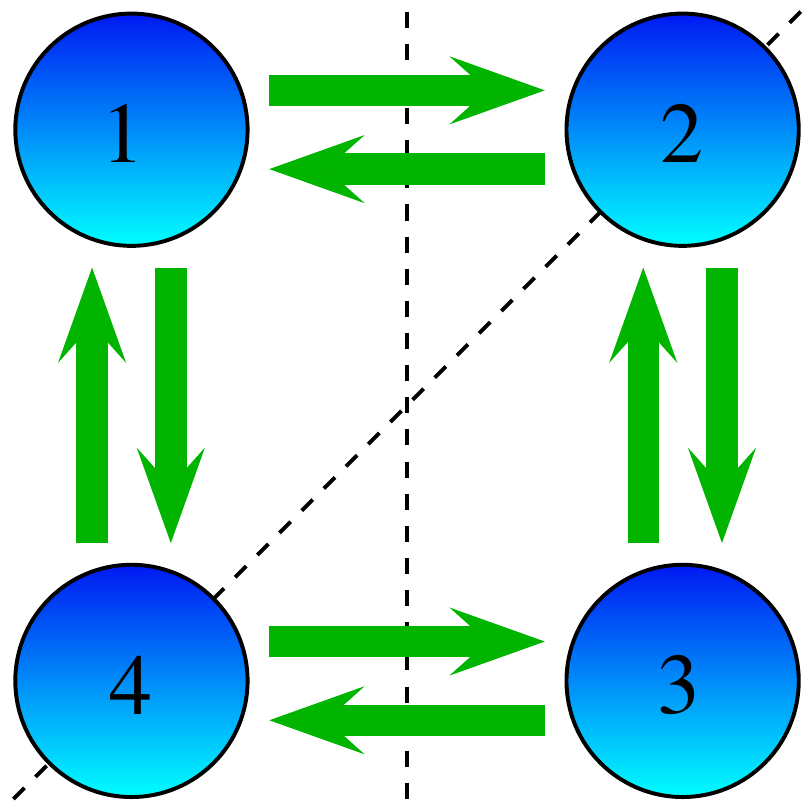} 
    \end{minipage}
    \caption{(a) Network representing the coupling structure corresponding to the equivariance property of the matrix $ A $ defined in Example~\ref{ex:network1}. The different colors of the arrows indicate the difference in the strength of the coupling to the nearest neighbors. These correspond to the coupling blocks $ C_1 $ and $ C_2 $ in the matrix $ A $. (b) $ D_4 $-symmetric network. The two axes of reflection are illustrated by dashed lines.}
    \label{fig:Z4_D4}
\end{figure}
\end{example}

\begin{example} \label{ex:Petersen}
Consider the Petersen graph $ \mathcal{P} $ shown in Figure~\ref{fig:Petersen}, cf.~\cite{HS93}. The system matrix has the following structure
\begin{equation*}
    A = 
    \begin{pmatrix}
        B & C & 0 & 0 & C & C & 0 & 0 & 0 & 0 \\
        C & B & C & 0 & 0 & 0 & C & 0 & 0 & 0 \\
        0 & C & B & C & 0 & 0 & 0 & C & 0 & 0 \\
        0 & 0 & C & B & C & 0 & 0 & 0 & C & 0 \\
        C & 0 & 0 & C & B & 0 & 0 & 0 & 0 & C \\
        C & 0 & 0 & 0 & 0 & B & 0 & C & C & 0 \\
        0 & C & 0 & 0 & 0 & 0 & B & 0 & C & C \\
        0 & 0 & C & 0 & 0 & C & 0 & B & 0 & C \\
        0 & 0 & 0 & C & 0 & C & C & 0 & B & 0 \\
        0 & 0 & 0 & 0 & C & 0 & C & C & 0 & B
    \end{pmatrix}.
\end{equation*}
Note that each vertex could again represent a $ d $-dimensional dynamical system. However, for simplicity, we will assume here that $ d = 1 $. The automorphism group of the Petersen graph is isomorphic to $ S_5 $, i.e.~$ \Aut(\mathcal{P}) \cong S_5 $. That is, there are $ 5! = 120 $ automorphisms. The group $ S_5 $ is generated by the transposition $ \pi_t = (1\;2) $ and the $ n $-cycle $ \pi_c = (1\;2\;3\;4\;5) $. Every permutation $ \pi \in S_5 $ acts as follows on the graph $ \mathcal{P} $: Each vertex with the assigned set $ \{a, b\} $ is mapped to the vertex with set $ \{\pi(a), \pi(b)\} $. Thus, $ \pi_t $ induces $ (3\;7)(4\;10)(8\;9) $, a transformation of the graph that fixes vertices $ 1, 2, 5, 6 $ and exchanges $ 3 \leftrightarrow 7 $, $ 4 \leftrightarrow 10 $, and $ 8 \leftrightarrow 9 $. The permutation $ \pi_c $ induces $ (1\;4\;2\;5\;3)(6\;9\;7\;10\;8) $, a rotation of the graph by $ 144 $ degrees. \exampleSymbol

\begin{figure}[htb]
    \centering
    \includegraphics[width=0.5\textwidth]{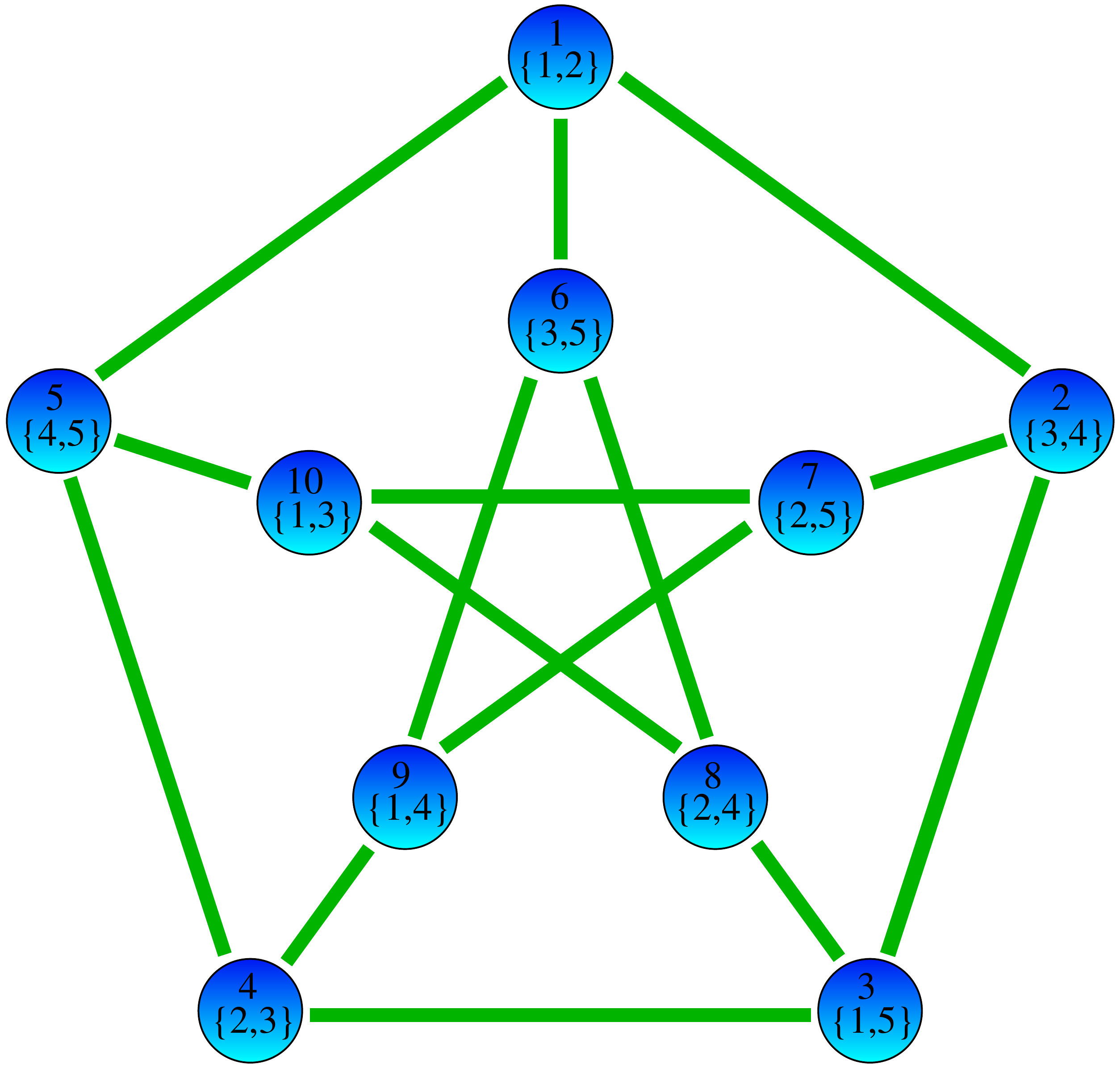} 
    \caption{Petersen graph $ \mathcal{P} $. The vertices of the graph are labeled $ 1, 2, \dots, 10 $. In addition to the vertex number, each vertex is assigned one of the subsets $ \{a, b\} \subset \{1, 2, 3, 4, 5\} $. Vertex $ i $ is then connected to vertex $ j $ if and only if the corresponding sets are disjoint, i.e.~$ \{a_i, b_i\} \cap \{a_j, b_j\} = \varnothing $.}
    \label{fig:Petersen}
\end{figure}
\end{example}

\subsection{Some Representation Theory}

We now briefly review some elementary concepts from group representation theory that are relevant within our framework. As above, let $ \Gamma \subset {\bf O}(n) $ be a finite group of orthogonal matrices.

\begin{definition}
A subspace $ V \subset \R^n $ is called \emph{$ \Gamma $-invariant} if $ \gamma v \in V $ for all $ v \in V $ and all $ \gamma \in \Gamma $. A $ \Gamma $-invariant subspace $ V $ is \emph{$ \Gamma $-irreducible} if it does not contain a proper nontrivial $ \Gamma $-invariant subspace.
\end{definition}

Observe that the operation of $ \Gamma $ on a $ \Gamma $-invariant subspace $V$ induces a corresponding \emph{representation}
\begin{equation}
    \vartheta_V : \Gamma \to {\bf L}(V), \quad \gamma\mapsto \gamma_{\vert V}
\end{equation}
of the group on that space. Here ${\bf L}(V)$ denotes the space of linear maps from $ V $ into $ V $. Representations of $ \Gamma $ induced by irreducible subspaces  are called \emph{irreducible representations} of $ \Gamma $. Up to isomorphism, a finite group possesses finitely many different irreducible representations. In fact, the irreducible representations of the finite groups are well known.

The \emph{character} $ \chi_\vartheta $ of an irreducible representation $ \vartheta $ assigns to each element of $ \vartheta $ its trace, that is
\begin{equation} \label{eq:character}
    \chi_\vartheta(\gamma) = \trace (\vartheta(\gamma))
\end{equation}
for all $ \gamma \in \Gamma $.

\begin{example}
In our guiding example -- the dihedral group $ \Gamma = D_4 $ acting on $ \R^n $, where $ n = 8 $ (i.e.~set $ d = 2 $ in Example~\ref{ex:network1}) -- there exist three different types of irreducible subspaces or representations, respectively. Here and in what follows $ \langle w^1, w^2, \ldots, w^k \rangle $ denotes the span of the vectors $ w^1, w^2, \ldots, w^k \in \R^n $.
\begin{itemize}
\item[(a)] $ \vartheta_1 : \Gamma \to \R^{1,1} $ is given by $ \vartheta_1(\gamma) = 1 $ for all $ \gamma \in \Gamma $. A corresponding irreducible subspace $ V_1\subset \R^n $ is given by $ \langle v^1\rangle $ with $ v^1 = (1, 0, 1, 0, 1, 0, 1, 0)^T $.
\item[(b)] $ \vartheta_2 : \Gamma \to \R^{1,1} $ is given by $ \vartheta_2(\gamma) = 1 $ for all $ \gamma \in \{R_0, R_2, S_1, S_3\} $ and $ \vartheta_2(\gamma) = -1 $ for all $ \gamma \in \{ R_1, R_3, S_2, S_4\} $. A corresponding irreducible subspace $ V_2 \subset \R^n $ is given by $ \langle v^2 \rangle $ with $ v^2 = (1, 0, -1, 0, 1, 0, -1, 0)^T $.
\item[(c)] $ \vartheta_3 : \Gamma \to \R^{2,2} $ is generated by 
\begin{equation*}
    \vartheta_3(R_1) =
    \begin{pmatrix}
        0 & -1 \\
        1 &  0
    \end{pmatrix}
    \quad \text{and} \quad 
    \vartheta_3(S_1) =
    \begin{pmatrix}
        1 & 0\\
        0 & -1
    \end{pmatrix}.
\end{equation*}
\end{itemize}
A corresponding irreducible subspace $ V_3 \subset \R^n $ is given by $ \langle v^3_1, v^3_2 \rangle $ with
\begin{equation*}
    v^3_1 = (0, 0, 1, 0, 0, 0, -1, 0)^T \quad \text{and} \quad v^3_2 = (-1, 0, 0, 0, 1, 0, 0, 0)^T.
    \tag*{\exampleSymbol}
\end{equation*}
\end{example}

\begin{remark}
\begin{itemize}
\item[(a)] We remark that the group $ D_4 $ possesses two additional irreducible representation, namely $ \vartheta_{4,5}: \Gamma \to \R^{1,1} $, given by
\begin{equation*}
    \vartheta_4(\gamma) = 1 \text{ for all } \gamma \in \{R_0, R_1, R_2, R_3\}, 
    \quad
    \vartheta_4(\gamma) = -1 \text{ for all } \gamma \in \{ S_1, S_2, S_3, S_4\}
\end{equation*}
and
\begin{equation*}
    \vartheta_5(\gamma) = 1 \text{ for all } \gamma \in \{R_0, R_2, S_2, S_4\},
    \quad
    \vartheta_5(\gamma) = -1 \text{ for all } \gamma \in \{ R_1, R_3, S_1, S_3\}.
\end{equation*}
In our network example, these are not relevant since corresponding irreducible subspaces do not exist in $ \R^n $ in this case. To see this, consider the representation $ \vartheta_4 $. Here a vector $ v \in \R^n \setminus \{0\} $ satisfying
\begin{equation*}
    \gamma v = v \text{ for all } \gamma \in \{R_0, R_1, R_2, R_3\}
\end{equation*}
has already full $ D_4 $-symmetry and therefore cannot satisfy $\gamma v = -v$ for all $\gamma \in  \{S_1, S_2, S_3, S_4\}$. The analogue argument is valid for the irreducible representation $\vartheta_5$.
\item[(b)] For the cyclic case $ C_1 \ne C_2 $ -- that is, $ \Gamma = {\mathbb Z}_4 $ acting on $ \R^n $ as shown in Figure~\ref{fig:Z4_D4}(a)~--, the four one-dimensional complex irreducible representations are generated by
\begin{equation*}
    \vartheta_1(R_1) = 1, \quad \vartheta_2(R_1) = -1, \quad \vartheta_3(R_1) = i, \quad \text{and} \quad \vartheta_4(R_1) = -i.
\end{equation*}
Over the real numbers, the corresponding two-dimensional irreducible representation of the cyclic group $ \mathbb{Z}_4 $ is
\begin{equation*}
    \widetilde{\vartheta}_3(R_1) =
    \begin{pmatrix}
        0 & -1 \\
        1 &  0
    \end{pmatrix}.
\end{equation*}
\end{itemize}
\end{remark}

Observe that eigenspaces of the matrix $ A $ are $ \Gamma $-invariant. This follows immediately from the equivariance of $ A $. That is, let $ v $ be an eigenvector of the matrix $ A $ with corresponding eigenvalue $ \lambda $, then
\begin{equation*}
    A v = \lambda v \Longrightarrow A(\gamma v) = \gamma A v = \lambda (\gamma v).
\end{equation*}
Thus, the existence of higher-dimensional irreducible representations naturally leads to the existence of repeated eigenvalues of the matrix $ A $. It is this observation which is of particular relevance for control and observation purposes and we will now further elaborate on this.

\begin{definition}
\begin{itemize}
\item[(a)] Let $ \vartheta $ be an irreducible representation of $\Gamma$ of dimension $n_\vartheta$.
Then $\vartheta $ is \emph{absolutely irreducible} if the only linear maps on $\R^{n_\vartheta}$ which commute with all $\gamma\in\Gamma$ are multiples of the identity.
\item[(b)]  A representation $\vartheta_W$ of $\Gamma$ is called \emph{$\Gamma$-simple} if either
\begin{itemize}
\item[(i)] $\vartheta_W$ is irreducible but not absolutely irreducible, or
\item[(ii)] $\vartheta_W$ is the sum of two isomorphic absolutely irreducible
representations $\vartheta_V$.
\end{itemize}
\end{itemize}
\end{definition}

\begin{example}
The irreducible representation $ \vartheta_3 $ of $ D_4 $ is absolutely irreducible. The irreducible representation $ \widetilde{\vartheta}_3 $ of $ \mathbb{Z}_4 $ is not absolutely irreducible. To see this, note that $ \mathbb{Z}_4 $ is a commutative group. \exampleSymbol
\end{example}

Using this terminology, we can now describe the generic structure of eigenspaces of equivariant matrices. This result is due to M.~Golubitsky and I.~Stewart and can be found in~\cite {GSS88}.

\begin{proposition} \label{prop:eigenstructure}
Let $ \lambda $ be an eigenvalue of a matrix $ A $ satisfying  the equivariance condition~\eqref{eq:Aeq}.
\begin{itemize}
\item[(i)] Suppose that $ \lambda \in \R $ and let $ W $ be the corresponding eigenspace. Then, generically, $W$ is absolutely irreducible.
\item[(ii)] Suppose that $ \lambda \in \C \setminus \R $ and let $ W\subset\R^n $ be the space spanned by real and imaginary parts of corresponding eigenvectors. Then, generically, $ W $ is $ \Gamma $-simple.
\end{itemize}
\end{proposition}

Observe that the first part of Proposition~\ref{prop:eigenstructure} implies that, generically, there will be no multiple eigenvalues (real or complex) in problems which possess $ \mathbb{Z}_k $-symmetry. This follows from the fact that the two-dimensional representation of $ \mathbb{Z}_k $ is not absolutely irreducible. However, we do expect to obtain multiple eigenvalues (real or complex) in networks which are $ D_k $-symmetric or which have the structure of the Petersen graph. 

\begin{definition}
Let $ \vartheta $ be an irreducible representation. Then we denote by $ W_\vartheta $ its \emph{isotypic component}, that is, the subspace of $ \R^n $ which consists of the sum of all irreducible subspaces isomorphic to $ \vartheta $.
\end{definition}

We are now in the position to state a fundamental result from group representation theory which is of importance for our investigations.

\begin{theorem}
We have
\begin{equation} \label{isodecom}
    \R^n = \bigoplus W_\vartheta,
\end{equation}
where the sum is taken over all irreducible representations of $ \Gamma $ acting on $ \R^n $. Moreover, each isotypic component is an invariant subspace for the matrix $ A $:
\begin{equation}
    A(W_\vartheta) \subset W_\vartheta.
\end{equation}
\end{theorem}

In representation theory, the decomposition~\eqref{isodecom} is called the \emph{isotypic decomposition}. Consequently, $ A $ possesses a block diagonal structure if we choose a basis according to this decomposition of $ \R^n $.

The projections onto the isotypic components are explicitly known. In fact, let $ \vartheta $ be an irreducible representation of dimension $ n_\vartheta $. Then the projection onto the corresponding isotypic component is given by (see \eqref{eq:character})
\begin{equation} \label{eq:projiso}
    P_\vartheta x = \frac{n_\vartheta}{\abs{\Gamma}} \sum_{\gamma \in \Gamma} \chi_\vartheta(\gamma) \gamma x.
\end{equation}

\begin{example} \label{ex:guide}
First we compute the projections $ P_i $ onto the isotypic components $ W_i $ for the irreducible representations $ \vartheta_i $, $ i = 1, 2, 3 $, of our guiding example. We obtain by using \eqref{eq:projiso}
\begin{equation*}
    P_1 = \frac{1}{4}
    \begin{pmatrix}[rrrr]
        I_d & I_d & I_d & I_d \\
        I_d & I_d & I_d & I_d \\
        I_d & I_d & I_d & I_d \\
        I_d & I_d & I_d & I_d
    \end{pmatrix},
    \quad
    P_2 = \frac{1}{4} 
    \begin{pmatrix}[rrrr]
         I_d & -I_d &  I_d & -I_d \\
        -I_d &  I_d & -I_d &  I_d \\
         I_d & -I_d &  I_d & -I_d \\
        -I_d &  I_d & -I_d &  I_d
    \end{pmatrix},
    \quad \text{and}
\end{equation*}
\begin{equation*}
    P_3 = \frac{1}{2} 
    \begin{pmatrix}[rrrr]
         I_d &    0 & -I_d &    0 \\
           0 &  I_d &    0 & -I_d \\
        -I_d &    0 &  I_d &    0 \\
           0 & -I_d &    0 &  I_d
    \end{pmatrix}.
\end{equation*}

The corresponding isotypic components are
\begin{equation*}
    \arraycolsep=1pt
    \begin{array}{lrrrrl}
        W_1 = \left\{\right.( & a, &  a, &  a, &  a &)^T \in \R^n : a   \in \R^d \left.\right\}, \\
        W_2 = \left\{\right.( & a, & -a, &  a, & -a &)^T \in \R^n : a   \in \R^d \left.\right\}, \\
        W_3 = \left\{\right.( & a, &  b, & -a, & -b &)^T \in \R^n : a,b \in \R^d \left.\right\}.
    \end{array}
\end{equation*}

For example, for $ n = 8 $, $ d = 2 $, and
\begin{equation*}
    A =
    \begin{pmatrix}[cc|cc|cc|cc]
        10  & -10  &  6 &   3 &  0 &   0 &  6 &   3 \\
         3  &  30  &  1 &   5 &  0 &   0 &  1 &   5 \\ \hline
         6  &   3  & 10 & -10 &  6 &   3 &  0 &   0 \\
         1  &   5  &  3 &  30 &  1 &   5 &  0 &   0 \\ \hline
         0  &   0  &  6 &   3 & 10 & -10 &  6 &   3 \\
         0  &   0  &  1 &   5 &  3 &  30 &  1 &   5 \\ \hline
         6  &   3  &  0 &   0 &  6 &   3 & 10 & -10 \\
         1  &   5  &  0 &   0 &  1 &   5 &  3 &  30
    \end{pmatrix},
\end{equation*}
we obtain with respect to a basis, composed from bases of $ W_1 $, $ W_2 $, and $ W_3 $,
\begin{equation*}
    \underbrace{
    \begin{pmatrix}[r]
        1 \\ 0 \\ 1 \\ 0 \\ 1 \\ 0 \\ 1 \\ 0
    \end{pmatrix},
    \begin{pmatrix}[r]
        0 \\ 1 \\ 0 \\ 1 \\ 0 \\ 1 \\ 0 \\ 1
    \end{pmatrix}}_{W_1},
    \underbrace{
    \begin{pmatrix}[r]
        1 \\ 0 \\ -1 \\ 0 \\ 1 \\ 0 \\ -1 \\ 0
    \end{pmatrix},
    \begin{pmatrix}[r]
        0 \\ 1 \\ 0 \\ -1 \\ 0 \\ 1 \\ 0 \\ -1
    \end{pmatrix}}_{W_2},
    \underbrace{
    \begin{pmatrix}[r]
        1 \\ 0 \\ 1 \\ 0 \\ -1 \\ 0 \\ -1 \\ 0
    \end{pmatrix},
    \begin{pmatrix}[r]
        -2 \\ 1 \\ 0 \\ 1 \\ 2 \\ -1 \\ 0 \\ -1
    \end{pmatrix},
    \begin{pmatrix}[r]
        1 \\ 0 \\ 0 \\ 0 \\ -1 \\ 0 \\ 0 \\ 0
    \end{pmatrix},
    \begin{pmatrix}[r]
        0 \\ 0 \\ 1 \\ -1 \\ 0 \\ 0 \\ -1 \\ 1
    \end{pmatrix}}_{W_3},
\end{equation*}
the following block diagonal structure for $ A $
\begin{equation*}
    \begin{pmatrix}
        22 & -4 & \omit \vrule &  0 &   0 & \omit        &  0 &  0 &  0 &  0 \\
         5 & 40 & \omit \vrule &  0 &   0 & \omit        &  0 &  0 &  0 &  0 \\ \cline{1-5}
         0 &  0 & \omit \vrule & -2 & -16 & \omit \vrule &  0 &  0 &  0 &  0 \\
         0 &  0 & \omit \vrule &  1 &  20 & \omit \vrule &  0 &  0 &  0 &  0 \\ \cline{3-10}
         0 &  0 & \omit        &  0 &   0 & \omit \vrule & 10 & -4 & -3 & -7 \\
         0 &  0 & \omit        &  0 &   0 & \omit \vrule &  3 & 24 &  3 &  0 \\
         0 &  0 & \omit        &  0 &   0 & \omit \vrule &  6 & 22 & 19 &  7 \\
         0 &  0 & \omit        &  0 &   0 & \omit \vrule &  0 & -6 &  3 & 27
    \end{pmatrix}.
    \tag*{\exampleSymbol}
\end{equation*}
\end{example}

\begin{remark}
If $ C_1 \ne C_2 $ in our guiding example, we obtain the complex isotypic components
\begin{equation*}
    \arraycolsep=1pt
    \begin{array}{lrrrrl}
        W_1 = \left\{\right.( & a, &    a, &  a, &   a)^T \in \C^n : a \in \R^d \left.\right\}, \\
        W_2 = \left\{\right.( & a, &   -a, &  a, &  -a)^T \in \C^n : a \in \R^d \left.\right\}, \\
        W_3 = \left\{\right.( & a, &  -ia, & -a, &  ia)^T \in \C^n : a \in \R^d \left.\right\}, \\
        W_4 = \left\{\right.( & a, &   ia, & -a, & -ia)^T \in \C^n : a \in \R^d \left.\right\}.
    \end{array}
\end{equation*}
\end{remark}

\begin{example}
For the Petersen graph introduced in Example~\ref{ex:Petersen}, we obtain the irreducible representations:
\begin{itemize}
\item[(a)] $ \vartheta_1 : \Gamma \to \R^{1,1} $ is again given by $ \vartheta_1(\gamma) = 1 $ for all $ \gamma \in \Gamma $.
\item[(b)] $ \vartheta_2 : \Gamma \to \R^{5,5} $ is generated by 
\begin{equation*}
    \vartheta_2(\pi_t) =
    \begin{pmatrix}[rrrrr]
        0 & 0 & 0 & 1 & 0 \\
        0 & 0 & 0 & 0 & 1 \\
        0 & 0 & 1 & 0 & 0 \\
        1 & 0 & 0 & 0 & 0 \\ 
        0 & 1 & 0 & 0 & 0 
    \end{pmatrix}
    \quad \text{and} \quad
    \vartheta_2(\pi_c) =
    \begin{pmatrix}[rrrrr]
        0 & -1 & -1 &  0 & -1 \\
       -1 &  1 &  1 & -1 &  1 \\
        1 &  0 & -1 &  0 & -1 \\
        0 &  1 &  0 &  0 &  0 \\
        0 & -1 & -1 &  1 &  0
    \end{pmatrix}.
\end{equation*}
\item[(c)] $ \vartheta_3 : \Gamma \to \R^{4,4} $ is generated by 
\begin{equation*}
    \vartheta_3(\pi_t) =
    \begin{pmatrix}[rrrr]
        -1 & 0 & 0 & 0 \\
        -1 & 1 & 0 & 0 \\
         0 & 0 & 1 & 0 \\
        -1 & 0 & 0 & 1
    \end{pmatrix}
    \quad \text{and} \quad
    \vartheta_3(\pi_c) =
    \begin{pmatrix}[rrrr]
        -1 & 1 &  0 &  0 \\
        -1 & 1 & -1 &  0 \\
         0 & 1 & -1 & -1 \\
        -1 & 0 &  0 &  0
    \end{pmatrix}.
\end{equation*}
Since $ n = 10 $, we can conclude that each irreducible representation occurs exactly once in the isotypic decomposition \eqref{isodecom} and therefore the eigenvalues of the adjacency matrix are $ \lambda_1 = 3 $, $ \lambda_2 = 1 $, and $ \lambda_3 = -2 $ with multiplicities $ 1 $, $ 5 $, and $ 4 $.
\end{itemize}

The irreducible representations were computed with the discrete algebra tool GAP~\cite{GAP15}. Note that the matrices are not orthogonal. Using \eqref{eq:projiso}, we compute the projections $ P_1 = \frac{1}{10} \mathbf{1} \mathbf{1}^T $,
\begin{equation*}
    P_2 = \frac{1}{6}
    \begin{pmatrix}[rrrrrrrrrr]
        3 &  1 & -1 & -1 &  1 &  1 & -1 & -1 & -1 & -1 \\
        1 &  3 &  1 & -1 & -1 & -1 &  1 & -1 & -1 & -1 \\
       -1 &  1 &  3 &  1 & -1 & -1 & -1 &  1 & -1 & -1 \\
       -1 & -1 &  1 &  3 &  1 & -1 & -1 & -1 &  1 & -1 \\
        1 & -1 & -1 &  1 &  3 & -1 & -1 & -1 & -1 &  1 \\
        1 & -1 & -1 & -1 & -1 &  3 & -1 &  1 &  1 & -1 \\
       -1 &  1 & -1 & -1 & -1 & -1 &  3 & -1 &  1 &  1 \\
       -1 & -1 &  1 & -1 & -1 &  1 & -1 &  3 & -1 &  1 \\
       -1 & -1 & -1 &  1 & -1 &  1 &  1 & -1 &  3 & -1 \\
       -1 & -1 & -1 & -1 &  1 & -1 &  1 &  1 & -1 &  3
    \end{pmatrix},
\end{equation*}
\begin{equation*}
    P_3 = \frac{1}{15}
    \begin{pmatrix}[rrrrrrrrrr]
        6 & -4 &  1 &  1 & -4 & -4 &  1 &  1 &  1 &  1 \\
       -4 &  6 & -4 &  1 &  1 &  1 & -4 &  1 &  1 &  1 \\
        1 & -4 &  6 & -4 &  1 &  1 &  1 & -4 &  1 &  1 \\
        1 &  1 & -4 &  6 & -4 &  1 &  1 &  1 & -4 &  1 \\
       -4 &  1 &  1 & -4 &  6 &  1 &  1 &  1 &  1 & -4 \\
       -4 &  1 &  1 &  1 &  1 &  6 &  1 & -4 & -4 &  1 \\
        1 & -4 &  1 &  1 &  1 &  1 &  6 &  1 & -4 & -4 \\
        1 &  1 & -4 &  1 &  1 & -4 &  1 &  6 &  1 & -4 \\
        1 &  1 &  1 & -4 &  1 & -4 & -4 &  1 &  6 &  1 \\
        1 &  1 &  1 &  1 & -4 &  1 & -4 & -4 &  1 &  6
    \end{pmatrix}.
    \tag*{\exampleSymbol}
\end{equation*}
\end{example}

We have already seen that equivariance properties lead to the existence of multiple eigenvalues in a generic way (Proposition~\ref{prop:eigenstructure}; cf.~\cite{GSS88, CM14}). For control purposes, it is important to analyze and understand this phenomenon in more detail. Here corresponding results by F\"assler and Stiefel \cite{FS92} and by Gatermann \cite{Ga96} are particularly useful. They have shown that the block diagonal structure induced by the isotypic components can even be refined further by introducing so-called \emph{symmetry adapted bases}.

Let us explain this result in detail. Let $ \vartheta $ be an irreducible representation of $\Gamma$. Then we define the following $ n_\vartheta $ projections by
\begin{equation} \label{eq:projsa}
    P^\mu_\vartheta x = \frac{n_\vartheta}{\abs{\Gamma}} \sum_{\gamma \in \Gamma}
    \left( \vartheta (\gamma^{-1})\right)_{\mu,\mu} \gamma x,
    \quad \mu = 1,\ldots,n_\vartheta.
\end{equation}
Observe that, if the corresponding unitary representations are used, all the projections defined above are orthogonal due to the fact that $ \Gamma \subset {\bf O}(n) $.

\begin{theorem} \label{thm:GW}
Let us denote the image of $ P^\mu_\vartheta $ by $ V^\mu_\vartheta $, $ \mu = 1, \ldots, n_\vartheta $. Then $ V^\mu_\vartheta $ is an $A$-invariant subspace of $ W_\vartheta $ for every $ \mu $ and, moreover, the restrictions $ A_{\vert  V^\mu_\vartheta} $ are all isomorphic for $ \mu = 1, \ldots, n_\vartheta $.
\end{theorem}

An immediate consequence of Theorem~\ref{thm:GW} is the following: Suppose that we make a change of coordinates and write $ A $ with respect to the symmetry adapted basis. Let $\vartheta$ be an absolutely irreducible representation of dimension $n_\vartheta$. Then there will be $ n_\vartheta $ identical blocks of dimension $\mbox{dim}(W_\vartheta) / n_\vartheta$ on the diagonal of $ A $, and generically these blocks will possess only simple eigenvalues.

\begin{example}
Let us consider again the system defined in Example~\ref{ex:guide} and compute the symmetry adapted basis for the action of $ \Gamma $, see~\eqref{eq:projsa}. The first two projections remain the same, that is $ P^1_1 = P_1 $ and $ P^1_2 = P_2 $, and for $ \vartheta_3 $ we compute
\begin{equation*}
    P_3^1 = \frac{1}{2} 
    \begin{pmatrix}
        0 & 0 & 0 & 0 \\
        0 & I_d & 0 & -I_d \\
        0 & 0 & 0 & 0 \\
        0 & -I_d & 0 & I_d
    \end{pmatrix}
    \quad \text{and} \quad 
    P_3^2 = \frac{1}{2} 
    \begin{pmatrix}
        I_d & 0 & -I_d & 0 \\
        0 & 0 & 0 & 0 \\
        -I_d & 0 & I_d & 0 \\
        0 & 0 & 0 & 0
    \end{pmatrix}.
\end{equation*}
With this information, we define the symmetry adapted basis
\begin{equation*}
    \underbrace{
    \begin{pmatrix}[r]
        1 \\ 0 \\ 1 \\ 0 \\ 1 \\ 0 \\ 1 \\ 0
    \end{pmatrix},
    \begin{pmatrix}[r]
        0 \\ 1 \\ 0 \\ 1 \\ 0 \\ 1 \\ 0\\ 1
    \end{pmatrix}}_{W_1},
    \underbrace{
    \begin{pmatrix}[r]
        1 \\ 0 \\ -1 \\ 0 \\ 1 \\ 0 \\ -1 \\ 0
    \end{pmatrix},
    \begin{pmatrix}[r]
        0 \\ 1 \\ 0 \\ -1 \\ 0 \\ 1 \\ 0 \\ -1
    \end{pmatrix}}_{W_2},
    \underbrace{
    \begin{pmatrix}[r]
        0 \\ 0 \\ 1 \\ 0 \\ 0 \\ 0 \\ -1 \\ 0 
    \end{pmatrix},
    \begin{pmatrix}[r]
        0 \\ 0 \\ 0 \\ 1 \\ 0 \\ 0 \\ 0 \\ -1
    \end{pmatrix},
    \begin{pmatrix}[r]
        1 \\ 0 \\ 0 \\ 0 \\ -1 \\ 0 \\ 0 \\ 0
    \end{pmatrix},
    \begin{pmatrix}[r]
        0 \\ 1 \\ 0 \\ 0 \\ 0 \\ -1 \\ 0 \\ 0 
    \end{pmatrix}}_{W_3}
\end{equation*}
and obtain the following block diagonal structure for $ A $ by a corresponding coordinate transformation:
\begin{equation*}
    \begin{pmatrix}
        22  & -4 & \omit \vrule &  0 &   0 & \omit        &  0 &   0 & \omit        &  0 &   0 \\
         5  & 40 & \omit \vrule &  0 &   0 & \omit        &  0 &   0 & \omit        &  0 &   0 \\ \cline{1-5}
         0  &  0 & \omit \vrule & -2 & -16 & \omit \vrule &  0 &   0 & \omit        &  0 &   0 \\
         0  &  0 & \omit \vrule &  1 &  20 & \omit \vrule &  0 &   0 & \omit        &  0 &   0 \\ \cline{3-9}
         0  &  0 & \omit        &  0 &   0 & \omit \vrule & 10 & -10 & \omit \vrule &  0 &   0 \\
         0  &  0 & \omit        &  0 &   0 & \omit \vrule &  3 &  30 & \omit \vrule &  0 &   0 \\ \cline{7-11}
         0  &  0 & \omit        &  0 &   0 & \omit        &  0 &   0 & \omit \vrule & 10 & -10 \\
         0  &  0 & \omit        &  0 &   0 & \omit        &  0 &   0 & \omit \vrule &  3 &  30
    \end{pmatrix}.
    \tag*{\exampleSymbol}
\end{equation*}
\end{example}

\begin{example}
For the Petersen graph example, refining these projections with the aid of \eqref{eq:projsa} leads to projections $ P_1^1 $, $ P_2^1, \dots, P_2^5 $, and $ P_3^1, \dots, P_3^4 $. Here, in order to obtain an orthonormal basis of $ \R^{10} $, we used the corresponding unitary representations. (Note that due to Weyl's unitary trick finite-dimensional representations of finite groups are unitarisable.) The subspaces $ V_{\vartheta_i}^\mu $ are spanned by the column vectors of the matrix
\begin{equation*}
    T =
    \begin{pmatrix}[rrrrrrrrrr]
        1/\sqrt{10} &   1/\sqrt{18} &   1/3 &  2/\sqrt{60} &  4/\sqrt{60} &     0 &            0 &   2/\sqrt{18} &   1/3 &   2/\sqrt{60} \\
        1/\sqrt{10} &   1/\sqrt{18} &   1/3 & -2/\sqrt{60} &  1/\sqrt{60} &   1/2 &            0 &  -2/\sqrt{18} &  -1/3 &   2/\sqrt{60} \\
        1/\sqrt{10} &  -1/\sqrt{18} &  -1/3 & -2/\sqrt{60} &  1/\sqrt{60} &   1/2 &   1/\sqrt{6} &   1/\sqrt{18} &   1/6 &  -3/\sqrt{60} \\
        1/\sqrt{10} &  -3/\sqrt{18} &     0 &            0 &            0 &     0 &  -1/\sqrt{6} &  -1/\sqrt{18} &   1/3 &   2/\sqrt{60} \\
        1/\sqrt{10} &  -1/\sqrt{18} &   1/6 &  5/\sqrt{60} &            0 &     0 &            0 &             0 &  -1/2 &  -3/\sqrt{60} \\
        1/\sqrt{10} &   1/\sqrt{18} &  -1/6 & -1/\sqrt{60} &  3/\sqrt{60} &  -1/2 &            0 &  -2/\sqrt{18} &   1/6 &  -3/\sqrt{60} \\
        1/\sqrt{10} &   1/\sqrt{18} &   1/3 & -2/\sqrt{60} & -4/\sqrt{60} &     0 &  -1/\sqrt{6} &   1/\sqrt{18} &   1/6 &  -3/\sqrt{60} \\
        1/\sqrt{10} &   1/\sqrt{18} &  -2/3 &            0 &            0 &     0 &  -1/\sqrt{6} &   1/\sqrt{18} &  -1/3 &   2/\sqrt{60} \\
        1/\sqrt{10} &  -1/\sqrt{18} &   1/6 & -3/\sqrt{60} & -1/\sqrt{60} &  -1/2 &   1/\sqrt{6} &   1/\sqrt{18} &  -1/3 &   2/\sqrt{60} \\
        1/\sqrt{10} &   1/\sqrt{18} &  -1/6 &  3/\sqrt{60} & -4/\sqrt{60} &     0 &   1/\sqrt{6} &  -1/\sqrt{18} &   1/3 &   2/\sqrt{60}
    \end{pmatrix}.
\end{equation*}
The matrix $ T $ transforms the adjacency matrix into a diagonal matrix whose entries are the eigenvalues of $ \mathcal{P} $. \exampleSymbol
\end{example}

\section{Controllability and Observability in Equivariant Systems}
\label{sec:Control Theory}

In this section, we will consider systems of the form
\begin{equation} \label{eq:LinearLTI}
    \begin{split}
        \dot{x} &= A \, x + B \, u, \\
        y       &= C \, x + D \, u,
    \end{split}
\end{equation}
where $ x \in \R^n $ is the state vector, $ u \in \R^p $ the control vector, and $ y \in \R^q $ the output vector. Furthermore, $ A \in \R^{n, n} $, $ B \in \R^{n, p} $, $ C \in \R^{q, n} $, and $ D \in \R^{q, p} $. The matrix $ D $ is often not required and can assumed to be zero. We will introduce only the basic definitions required for our analysis, for a more detailed description see for example~\cite{TSH01}.

Controllability and observability of such systems play an important role in modern control system theory. Controllability can be defined as follows: If for any given state $ x_0 $ at time $ t_0 $, there exist a finite time $ t_1 $ and a control function $ u(t) $ that transfers the state $ x(t_0) = x_0 $ to any state $ x(t_1) = x_1 $, then the system is said to be completely controllable. Similarly, a system is defined to be completely observable if for any time $ t_0 $ the initial state $ x(t_0) $ can be determined by observing the system output $ y(t) $ over a finite time interval $ t_0 < t < t_1 $.

\begin{definition}
Define
\begin{equation}
    P =
    \begin{pmatrix}
        B, & A B, & A^2 B, & \cdots, & A^{n-1} B
    \end{pmatrix}
\end{equation}
to be the controllability matrix and
\begin{equation}
    Q =
    \begin{pmatrix}
        C \\ C A \\ C A^2 \\ \vdots \\ C A^{n-1}
    \end{pmatrix}
\end{equation}
the observability matrix.
\end{definition}

It is well known that a linear time-invariant system of form~\eqref{eq:LinearLTI} is completely controllable if and only if the matrix $ P $ has full rank, i.e.~$ \rank(P) = n $, and completely observable if and only if $ Q $ has full rank, i.e.~$ \rank(Q) = n $. These rank conditions are typically used to determine whether a given system is controllable and observable. Another way to test for controllability and observability is to compute the rank of the matrices
\begin{equation}
    \widetilde{P} =
    \begin{pmatrix}
        s I - A, & B
    \end{pmatrix}
    \quad \text{and} \quad
    \widetilde{Q} =
    \begin{pmatrix}
        s I - A \\
        C
    \end{pmatrix}.
\end{equation}
System~\eqref{eq:LinearLTI} is completely controllable if and only if $ \rank(\widetilde{P}) = n $ and completely observable if and only if $ \rank(\widetilde{Q}) = n $ for all $ s \in \mathbb{C} $. If one of these two equivalent controllability conditions is satisfied, then we call the matrix pair $ (A, B) $ \emph{controllable}. Analogously, we call  $ (A, C) $ \emph{observable} if one of the equivalent observability conditions is satisfied.

Our considerations will be based on the following well known result (see e.g.~\cite{So90} or also \cite{CM14}): 

\begin{lemma} \label{lem:eigenvalue}
If $ (A, B) $ is controllable and $ \rank(B) = q $, then the geometric multiplicity of each eigenvalue of $ A $ is at most $ q $.
\end{lemma}

We now assume that $ A $ is $ \Gamma $-equivariant (see \eqref{eq:Aeq}). Let $ \vartheta_1, \dots, \vartheta_K $ be the irreducible representations of $ \Gamma $, denote by $ n_i $ their respective dimensions $ (i = 1, \ldots, K) $, and by $ P_{\vartheta_i} $ the projection onto the corresponding isotypic component (see \eqref{eq:projiso}). For simplicity, we assume that these representations are absolutely irreducible. The case where $ \Gamma $ possesses irreducible representations which are not absolutely irreducible can be treated in an analogous way.

By a combination of Lemma~\ref{lem:eigenvalue} with Theorem \ref{thm:GW} we obtain our first result.

\begin{proposition} \label{prop:rank}
Suppose that $(A, B)$ is controllable with $ \rank(B) = q $. Then
\begin{equation}
    q \ge N_\Gamma,
\end{equation}
where
\begin{equation}\label{eq:NGamma}
    N_\Gamma = \max_{i=1,\ldots,K}\{ n_i : P_{\vartheta_i} \neq 0\}.
\end{equation}
\end{proposition}

\begin{proof}
Observe that the condition $ P_{\vartheta_i} \neq 0 $ guarantees that the $ i^\text{th} $ irreducible representation of $ \Gamma $ actually occurs nontrivially in the isotypic decomposition \eqref{isodecom}. Therefore, by Theorem~\ref{thm:GW} the matrix $ A $ possesses an eigenvalue of geometric multiplicity at least $ N_\Gamma $. Now the result follows with Lemma~\ref{lem:eigenvalue}. 
\end{proof}
 
\begin{example}
In our guiding example $ N_\Gamma = 2 $, and therefore the rank of $ B $ has to be at least two. For the nonsymmetric case $ C_1 \ne C_2 $, however, $ N_\Gamma = 1 $ and $ \rank(B) \ge 1 $ is sufficient. For the Petersen graph example $ N_\Gamma = 5 $, which requires $\rank(B) \ge 5 $ in order to achieve controllability. \exampleSymbol
\end{example}

From now on, we will assume that $ P_{\vartheta_i} \neq 0 $ for $ i = 1, \ldots, K $. That is, we will restrict our attention to the irreducible representations which actually possess nontrivial isotypic components. We will now refine the result in Proposition~\ref{prop:rank} and address the question how to use the symmetry adapted basis for the identification of appropriate matrices $B$. That is, we would like to identify a control such that both the number of columns of $ B $ and the number of nonzero entries of $ B $ are as small as possible.

Let us denote by $ P^\mu_{\vartheta_i} $, $ i = 1, \ldots, K $, $ \mu = 1, \ldots, n_i $, the projections onto the symmetry adapted basis and by  $ V^\mu_{\vartheta_i} $ the corresponding subspaces (see Theorem~\ref{thm:GW}). Since the projections $ P^\mu_{\vartheta_i} $ are orthogonal, we know that
\begin{equation}
    \R^n = \bigoplus V^\mu_{\vartheta_i}
\end{equation}
is an orthogonal sum of the $ A $-invariant subspaces $ V^\mu_{\vartheta_i} $. We denote by
\begin{equation}\label{eq:Bi}
    {\cal B}^1_1, \ldots, {\cal B}^{n_1}_1, {\cal B}^1_2, \ldots, {\cal B}^{n_2}_2, \ldots, {\cal B}^1_K, \ldots, {\cal B}^{n_K}_K
\end{equation}
a corresponding orthonormal basis of $ \R^n $ where $ {\cal B}^\mu_i $ is an orthonormal basis of $ V^\mu_{\vartheta_i} $ of length $ d_i $. The coordinate transformation induced by this basis is called $ T $.

\begin{example}
Considering Example~\ref{ex:guide}, we see that
\begin{equation*}
    \underbrace{
    \begin{pmatrix}
        1/2 \\ 0 \\ 1/2 \\ 0 \\ 1/2 \\ 0 \\ 1/2 \\ 0
    \end{pmatrix},
    \begin{pmatrix}
        0 \\ 1/2 \\ 0 \\ 1/2 \\ 0 \\ 1/2 \\ 0\\ 1/2
    \end{pmatrix}}_{{\cal B}^1_1},
    \underbrace{
    \begin{pmatrix}
        1/2 \\ 0 \\ -1/2 \\ 0 \\ 1/2 \\ 0 \\ -1/2 \\ 0
    \end{pmatrix},
    \begin{pmatrix}
        0 \\ 1/2 \\ 0 \\ -1/2 \\ 0 \\ 1/2 \\ 0 \\ -1/2
    \end{pmatrix}}_{{\cal B}^1_2},
    \underbrace{
    \begin{pmatrix}
        0 \\ 0 \\ 1/\sqrt{2} \\ 0 \\ 0 \\ 0 \\ -1/\sqrt{2} \\ 0 
    \end{pmatrix},
    \begin{pmatrix}
        0 \\ 0 \\ 0 \\ 1/\sqrt{2} \\ 0 \\ 0 \\ 0 \\ -1/\sqrt{2}
    \end{pmatrix}}_{{\cal B}^1_3},
    \underbrace{
    \begin{pmatrix}
        1/\sqrt{2} \\ 0 \\ 0 \\ 0 \\ -1/\sqrt{2} \\ 0 \\ 0 \\ 0
    \end{pmatrix},
    \begin{pmatrix}
        0 \\ 1/\sqrt{2} \\ 0 \\ 0 \\ 0 \\ -1/\sqrt{2} \\ 0 \\ 0 
    \end{pmatrix}}_{{\cal B}^2_3}
\end{equation*}
yields the orthonormal basis of $ \R^8 $ and thus
\begin{equation}
    T =
    \begin{pmatrix}[cccccccc]
        1/2  &   0  &  1/2 &    0 &           0 &           0 &  1/\sqrt{2} &           0 \\
          0  & 1/2  &    0 &  1/2 &           0 &           0 &           0 &  1/\sqrt{2} \\
        1/2  &   0  & -1/2 &    0 &  1/\sqrt{2} &           0 &           0 &           0 \\
          0  & 1/2  &    0 & -1/2 &           0 &  1/\sqrt{2} &           0 &           0 \\
        1/2  &   0  &  1/2 &    0 &           0 &           0 & -1/\sqrt{2} &           0 \\
          0  & 1/2  &    0 &  1/2 &           0 &           0 &           0 & -1/\sqrt{2} \\
        1/2  &   0  & -1/2 &    0 & -1/\sqrt{2} &           0 &           0 &           0 \\
          0  & 1/2  &    0 & -1/2 &           0 & -1/\sqrt{2} &           0 &           0
    \end{pmatrix}.
    \tag*{\exampleSymbol}
\end{equation}
\end{example}

\begin{proposition}
Suppose that $ (A, B) $ is controllable. Then for each irreducible representation $ \vartheta_i $ of $ \Gamma $ there exist $ n_i $ columns $ b_1, \ldots, b_{n_i} $ of $ B $ such that
\begin{equation}
    P^\mu_{\vartheta_i} b_\mu \neq 0 \text{ for } \mu = 1, \ldots, n_i.
\end{equation}
\end{proposition}

\begin{proof}
Consider the block diagonal matrix
\begin{equation}
    \widetilde{A} = T^{-1} A T = T^T A T.
\end{equation}
A necessary condition for $ (\widetilde{A}, \widetilde{B}) $ to be controllable is that for each irreducible representation $ \vartheta_i $ there are $ n_i $ columns of $ \widetilde{B} $ with nontrivial components addressing all the $ n_i $ identical blocks corresponding to $ \vartheta_i $. More precisely, suppose that these $n_i$ identical blocks correspond to the index sets $I_i^\mu \subset \{ 1,\ldots,n\} $, $\mu = 1,\ldots,n_i$. Then for each $ \mu $ there is a column $ \tilde{b}^\mu $ of $ \widetilde{B} $ such that $ \sum_{j\in I_i^\mu} \vert \tilde b^\mu_j \vert \not = 0 $. Observing that
\begin{equation}
    B = T \widetilde{B},
\end{equation}
we compute for each such column $ \tilde{b}^\mu $ and $ b_\mu = T \tilde{b}^\mu $
\begin{equation}
    P^\mu_{\vartheta_i} b_\mu = P^\mu_{\vartheta_i} T \tilde{b}^\mu = S_i^\mu \tilde{b}^\mu \neq 0.
\end{equation}
Here, $ S_i^\mu $ denotes the matrix with columns $ [ {\bf 0} \cdots  {\bf 0}\, {\cal B}^\mu_i \, {\bf 0}\cdots {\bf 0} ] $.
\end{proof}

\begin{remark}
Consider the generic case, a matrix $ A $ with $ n $ distinct eigenvalues. For such a system, in order to be controllable, the matrix $ B $ must also address all eigenspaces. Assume that $ B = \sum_{i=1}^n \alpha_i v_i $ and that there is a $ j $ with $ \alpha_j = 0 $, then
\begin{equation*}
    \rank\left(\lambda_j I - A, \sum_{i \ne j} \alpha_i v_i \right) < n  
\end{equation*}
since $ v_j \in \mathop{ker}(\lambda_j I - A) $ and $ v_j $ cannot be generated by the other eigenvectors. Thus $ (A, B) $ is not controllable.
\end{remark}

With the following proposition we identify particular simple choices for $ B $.

\begin{proposition} \label{prop:em}
Recall that $d_i$ is the length of the orthonormal basis $ {\cal B}^\mu_i $ of $ V^\mu_{\vartheta_i} $ (see \eqref{eq:Bi}). Consider the $ m^\text{th} $ row of the transformation $ T = (t_{rs}) $ and assume $ t_{mj} \neq 0 $ where
\begin{equation}
    j = \sum_{\ell=1}^{i-1} n_\ell d_\ell + (\mu - 1) d_i + k
\end{equation}
for $\mu\in\{ 1,\ldots,n_i\}$ and a $k\in\{ 1,\ldots,d_i\}$. Then
\begin{equation}
    P^\mu_{\vartheta_i} e_m \neq 0,
\end{equation}
where $ e_m \in\R^n $ is the $ m^\text{th} $ canonical unit vector in $ \R^n $.
\end{proposition}

\begin{proof}
The assumption on the index $j$ guarantees that the $m^\text{th}$ column $ T_m $ of $ T^T $ contains a non-vanishing component corresponding to the invariant subspace $ V^\mu_{\vartheta_i} $ in the block diagonal structure of $\widetilde{A} = T^T A T$. Therefore
\begin{equation*}
    P^\mu_{\vartheta_i} e_m = P^\mu_{\vartheta_i} T T_m = S_i^\mu T_m \neq 0.
\end{equation*}
As before, $ S_i^\mu $ denotes the matrix with columns $ [ {\bf 0} \cdots  {\bf 0}\, {\cal B}^\mu_i \, {\bf 0}\cdots {\bf 0} ] $.
\end{proof}

With this proposition, we are now in the position to propose the following algorithm for the construction of the matrix $ B $. Let $ T = (t_{rs}) $ be the coordinate transformation matrix whose column vectors, denoted $ T_s $, form a basis.
\begin{equation} \label{alg:B}
\begin{minipage}{0.8\textwidth} \fontsize{10}{15} \selectfont \begin{algorithmic}[1]
\State Set $ V = \varnothing $, define $ B = B(V) $ to be the matrix whose columns are the unit vectors defined by the indices  $ V $.
\State Choose an absolutely irreducible representation $ \vartheta_i $ of $ \Gamma $ with $ n_i = N_\Gamma $, see~\eqref{eq:NGamma}.
\State Let $ T_s $ be the first column of the matrix $ T $ belonging to $ \vartheta_i $.
\State Find the first nonzero entry $ [T_{s}]_r $ in the vector $ T_s $, i.e.~$ [T_{s}]_r = t_{rs} \ne 0 $, with $ r \notin V $.
\State $ V \leftarrow V \cup \{ r \} $  and $ s \leftarrow s + d_i $.
\State If $ \abs{V} < n_i, $ go to step 4. (Count only indices added for representation $ \vartheta_i $.)
\State If $ \rank(P) < n $ for $ A $ and $ B $, choose the next largest representation $ \vartheta_i $ and go to step 3.
\end{algorithmic} \end{minipage}
\end{equation}

In the worst case, the algorithm would add all $ n $ unit vectors so that $ (A, B) $ is controllable. Instead of adding unit vectors directly to the matrix $ B $, one could check whether adding that vector would increase the rank of $ P $ in order to avoid redundant information in $ B $. Let us now apply this algorithm to the guiding example.

\begin{example}
In this case, Proposition~\ref{prop:em} and the structure of $ T $ imply that
\begin{equation}
    P^1_{\vartheta_3} e_3 \neq 0 \quad \text{and} \quad P^2_{\vartheta_3} e_1 \neq 0. 
\end{equation}
In fact, it can easily be verified that $ (A, B) $ is controllable with the choice
\begin{equation}
    B =
    \begin{pmatrix}[cc]
         1 & 0 \\
         0 & 0 \\
         0 & 1 \\
         0 & 0 \\
         0 & 0 \\
         0 & 0 \\
         0 & 0 \\
         0 & 0 
    \end{pmatrix}.
\end{equation}
Here, Algorithm~\ref{alg:B} would start with column $ 5 $ of $ T $, add $ e_3 $ to $ B $, then select column $ 7 $, and add $ e_1 $. Since the system is controllable with that choice of $ B $, the algorithm stops.
\exampleSymbol
\end{example}

\begin{remark}
Interestingly, for the network with $ \mathbb{Z}_4 $ symmetry, one vector, e.g.~$ B = (1, 0, 0, 0, 0, 0, 0, 0)^T $, suffices to render the system controllable. That is, the system is controllable from any node.
\end{remark}

\begin{example}
Similarly, using Algorithm~\ref{alg:B}, we obtain for the Petersen graph example and the representation $ \vartheta_2 $
\begin{equation*}
    P^1_{\vartheta_2} e_1 \ne 0, \quad
    P^2_{\vartheta_2} e_2 \ne 0, \quad
    P^3_{\vartheta_2} e_3 \ne 0, \quad
    P^4_{\vartheta_2} e_6 \ne 0, \quad
    P^5_{\vartheta_2} e_9 \ne 0.
\end{equation*}
The corresponding matrix
\begin{equation*}
    B =
    \begin{pmatrix}
        1 & 0 & 0 & 0 & 0 \\
        0 & 1 & 0 & 0 & 0 \\
        0 & 0 & 1 & 0 & 0 \\
        0 & 0 & 0 & 0 & 0 \\
        0 & 0 & 0 & 0 & 0 \\
        0 & 0 & 0 & 1 & 0 \\
        0 & 0 & 0 & 0 & 0 \\
        0 & 0 & 0 & 0 & 0 \\
        0 & 0 & 0 & 0 & 1 \\
        0 & 0 & 0 & 0 & 0
    \end{pmatrix}
\end{equation*}
also addresses the other blocks corresponding to $ \vartheta_1 $ and $ \vartheta_3 $. Thus, $ (A, B) $ is controllable. Not all configurations with $ N_\Gamma = 5 $ input vectors are controllable. Figure~\ref{fig:Petersen_cnc} shows different configurations which are controllable (controlled vertices are marked green) and noncontrollable (controlled vertices are marked red). \exampleSymbol

\begin{figure}[htb]
    \begin{minipage}[t]{0.19\textwidth}
        \includegraphics[width=\textwidth]{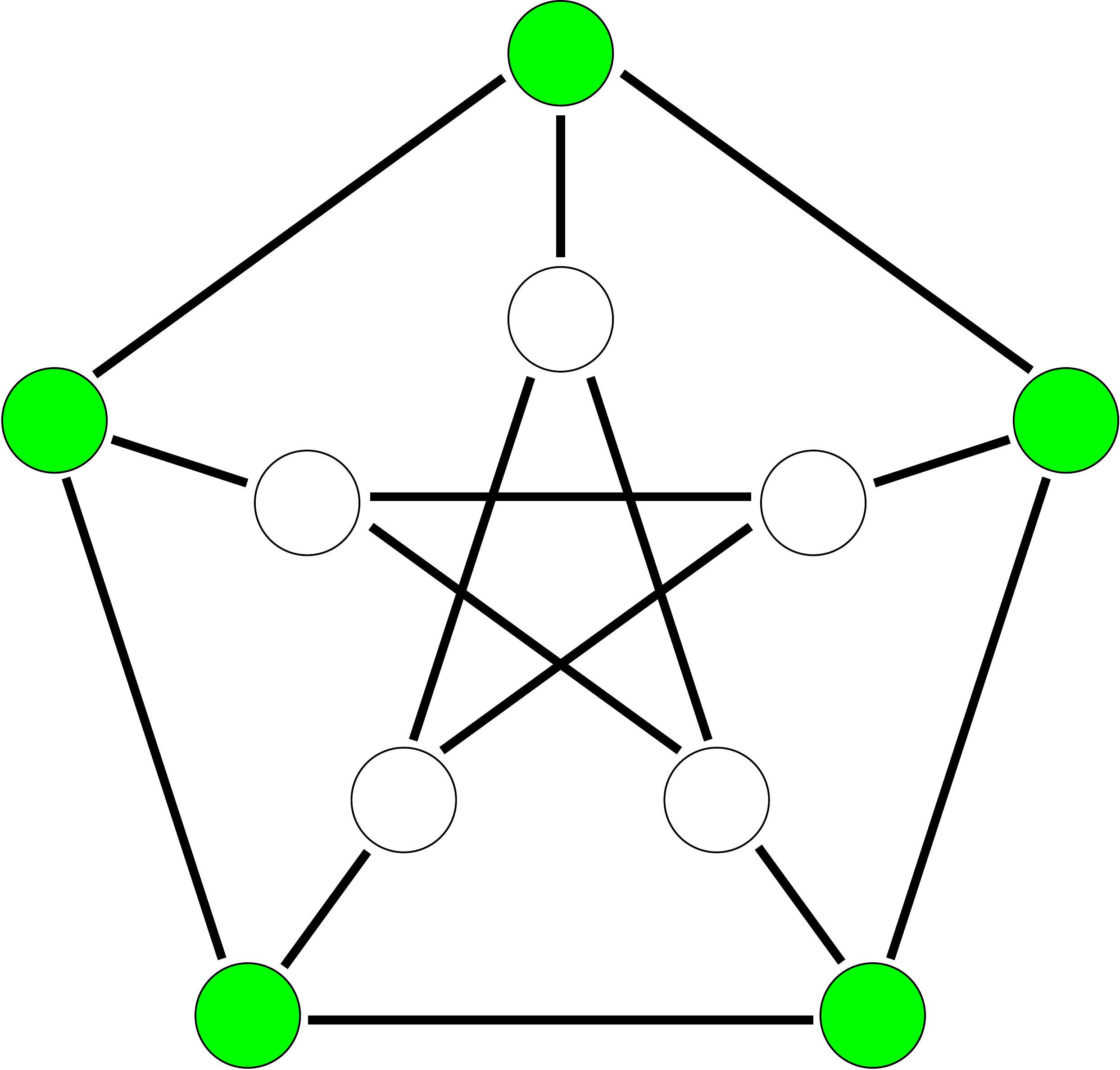}
    \end{minipage}
    \begin{minipage}[t]{0.19\textwidth}
        \includegraphics[width=\textwidth]{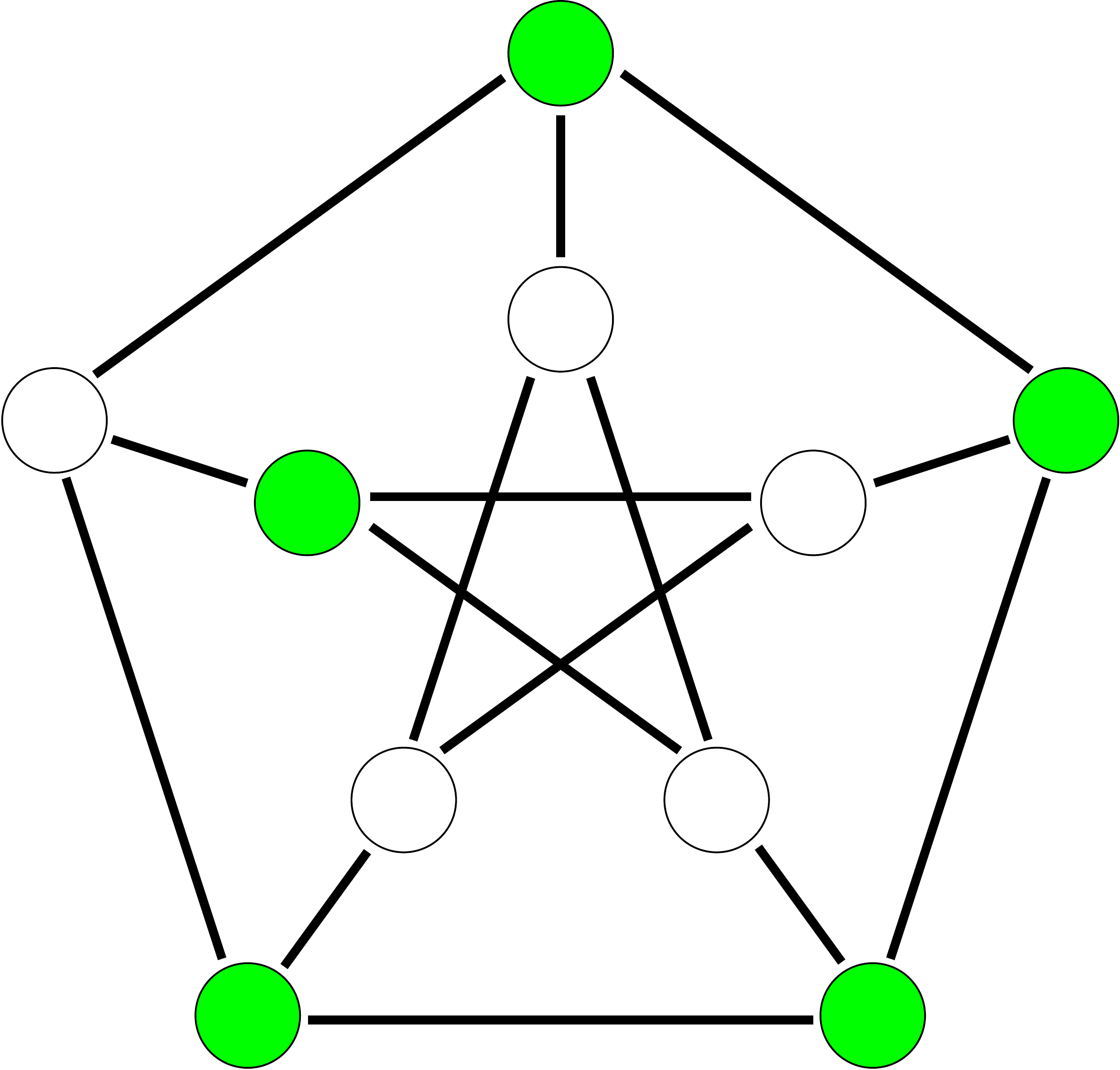}
    \end{minipage}
    \begin{minipage}[t]{0.19\textwidth}
        \includegraphics[width=\textwidth]{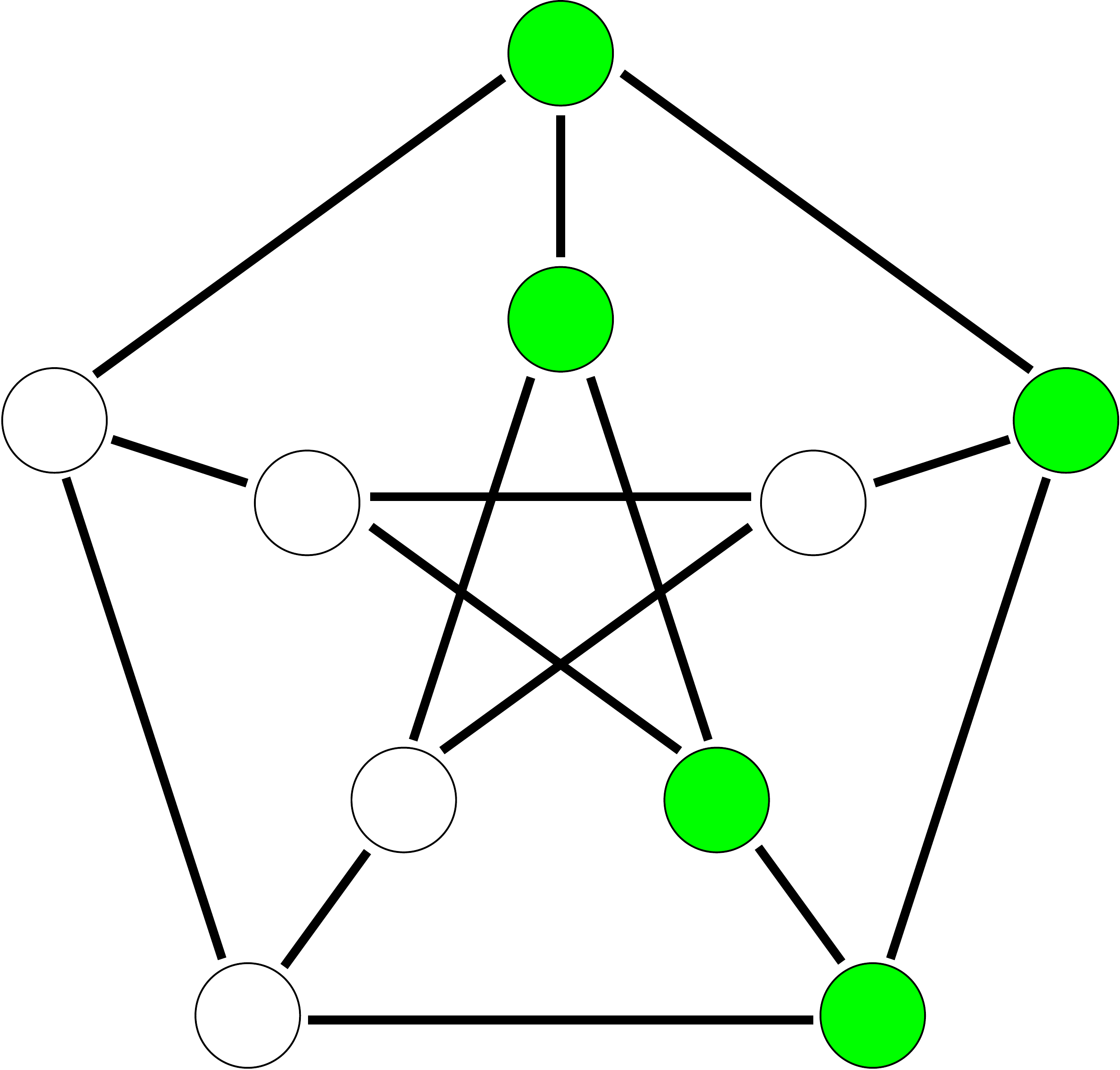}
    \end{minipage}
    \begin{minipage}[t]{0.19\textwidth}
        \includegraphics[width=\textwidth]{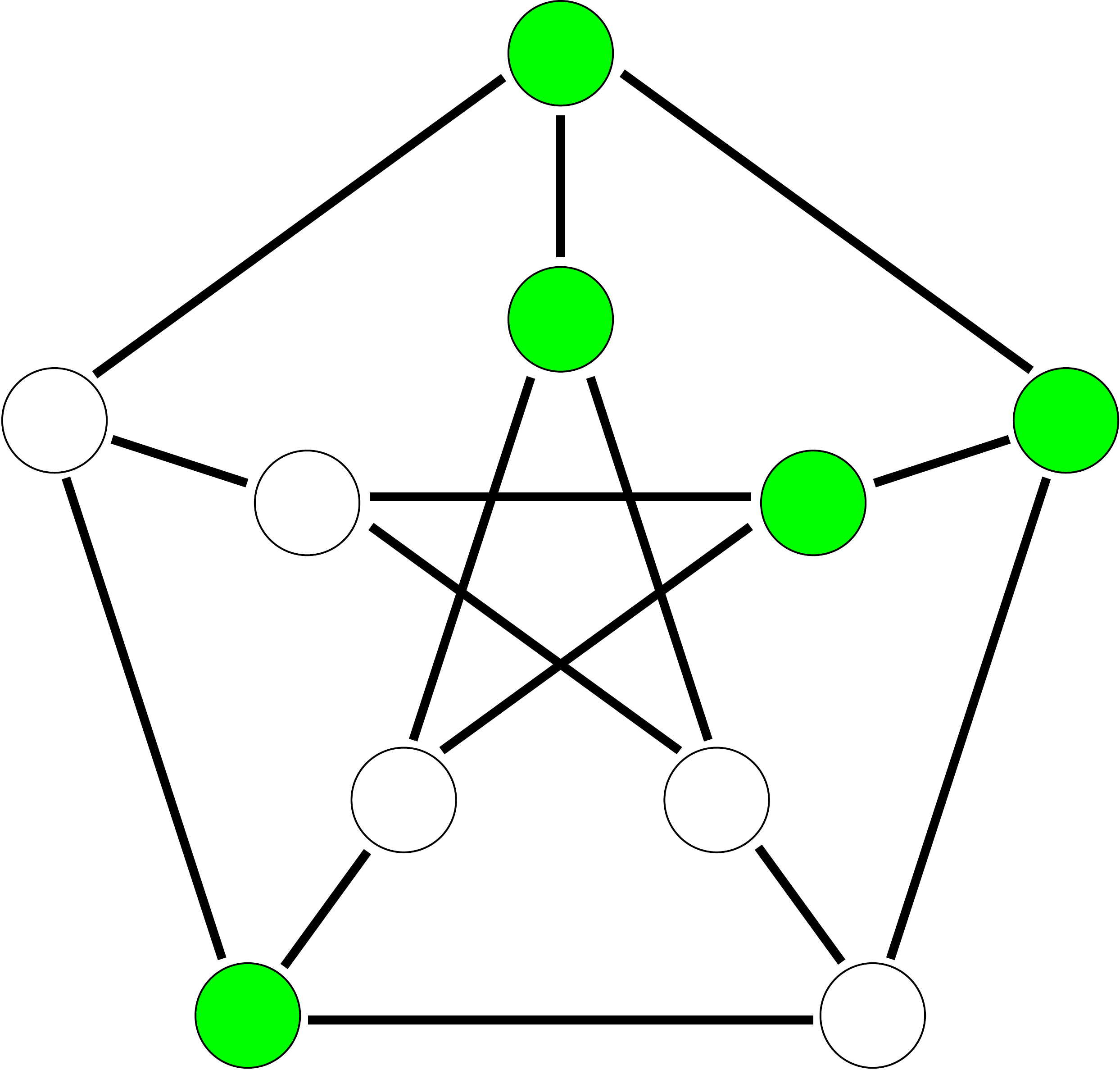}
    \end{minipage}
    \begin{minipage}[t]{0.19\textwidth}
        \includegraphics[width=\textwidth]{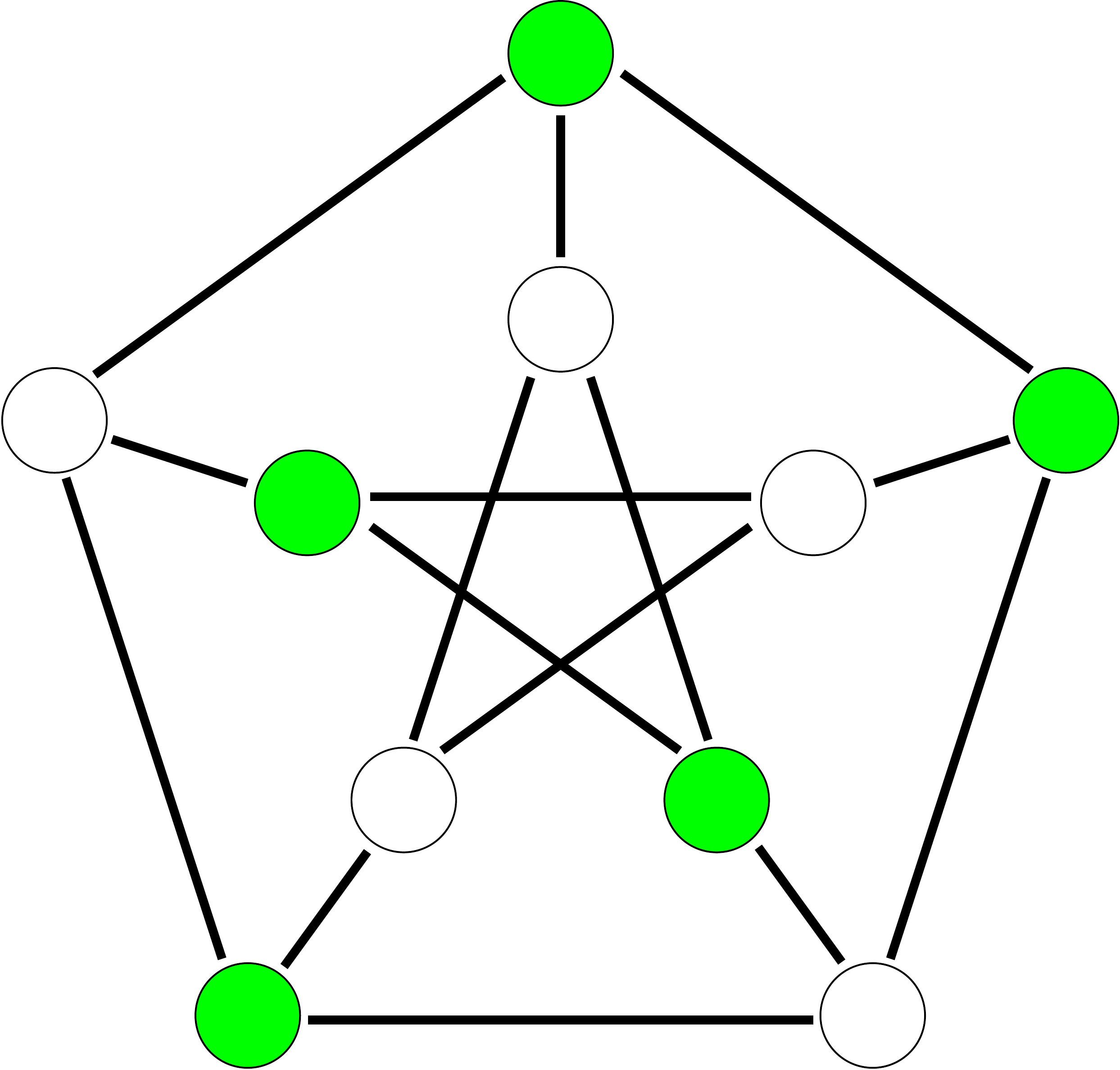}
    \end{minipage} \\[1em]
    \begin{minipage}[t]{0.19\textwidth}
        \includegraphics[width=\textwidth]{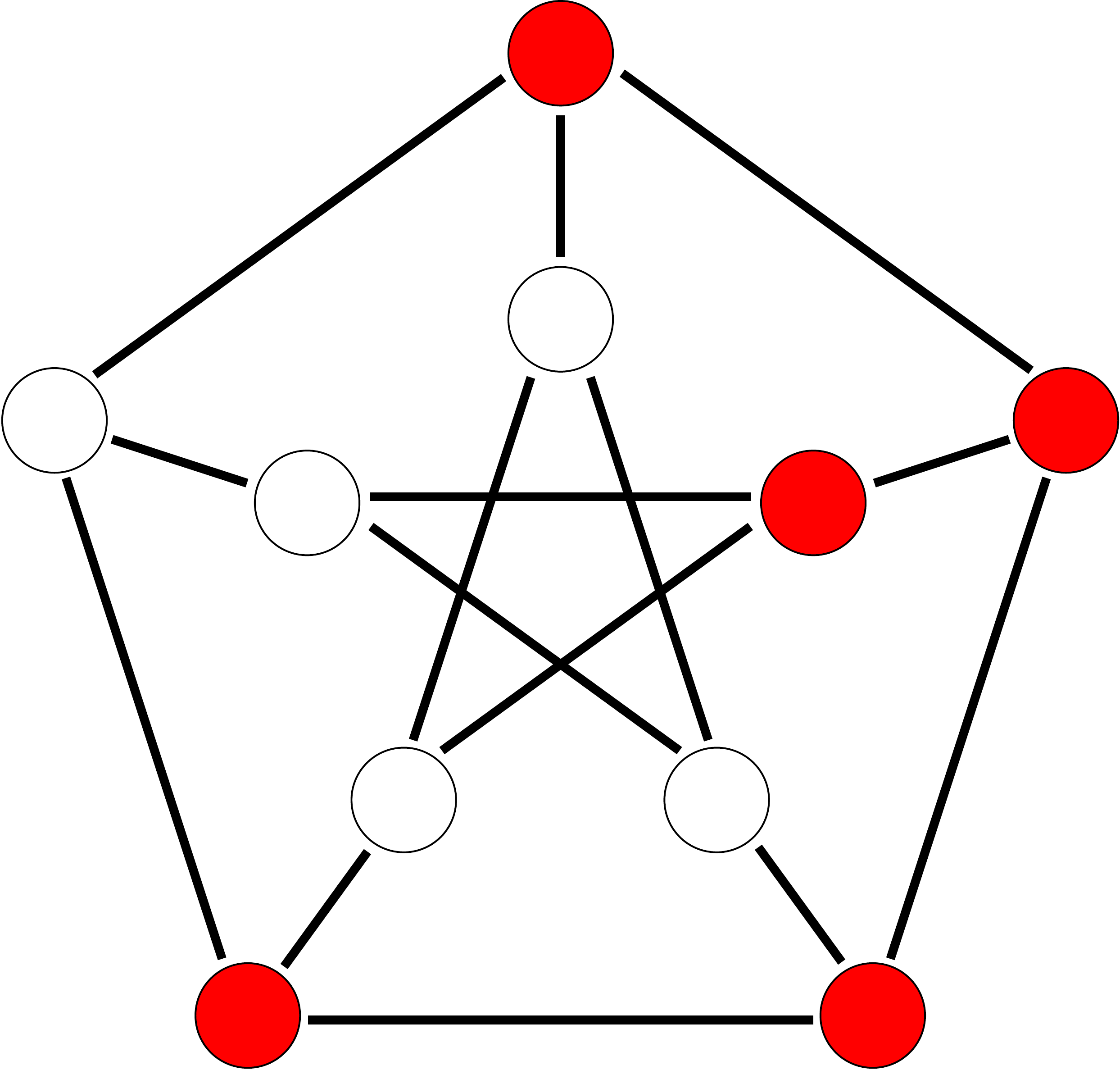}
    \end{minipage}
    \begin{minipage}[t]{0.19\textwidth}
        \includegraphics[width=\textwidth]{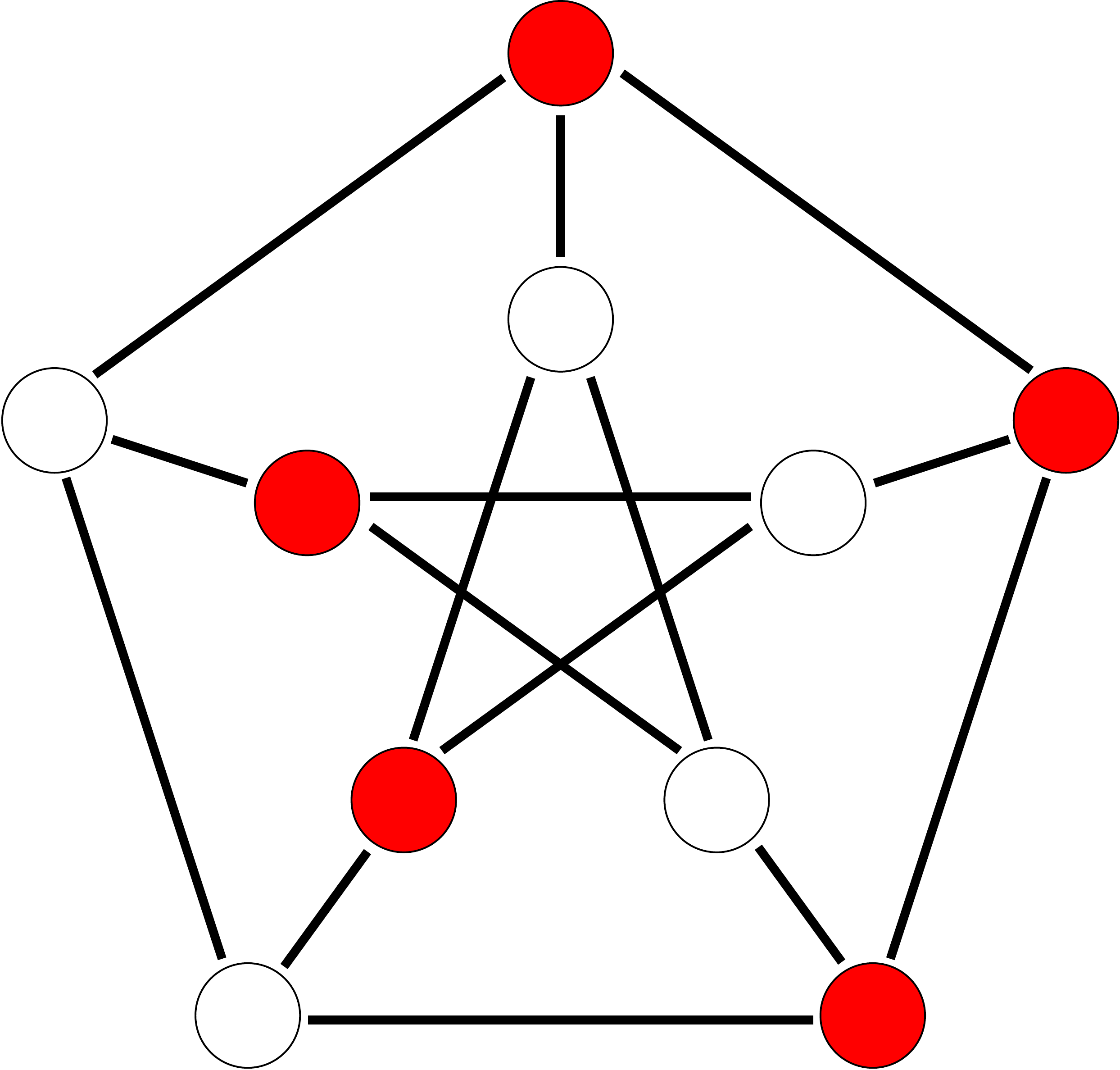}
    \end{minipage}
    \begin{minipage}[t]{0.19\textwidth}
        \includegraphics[width=\textwidth]{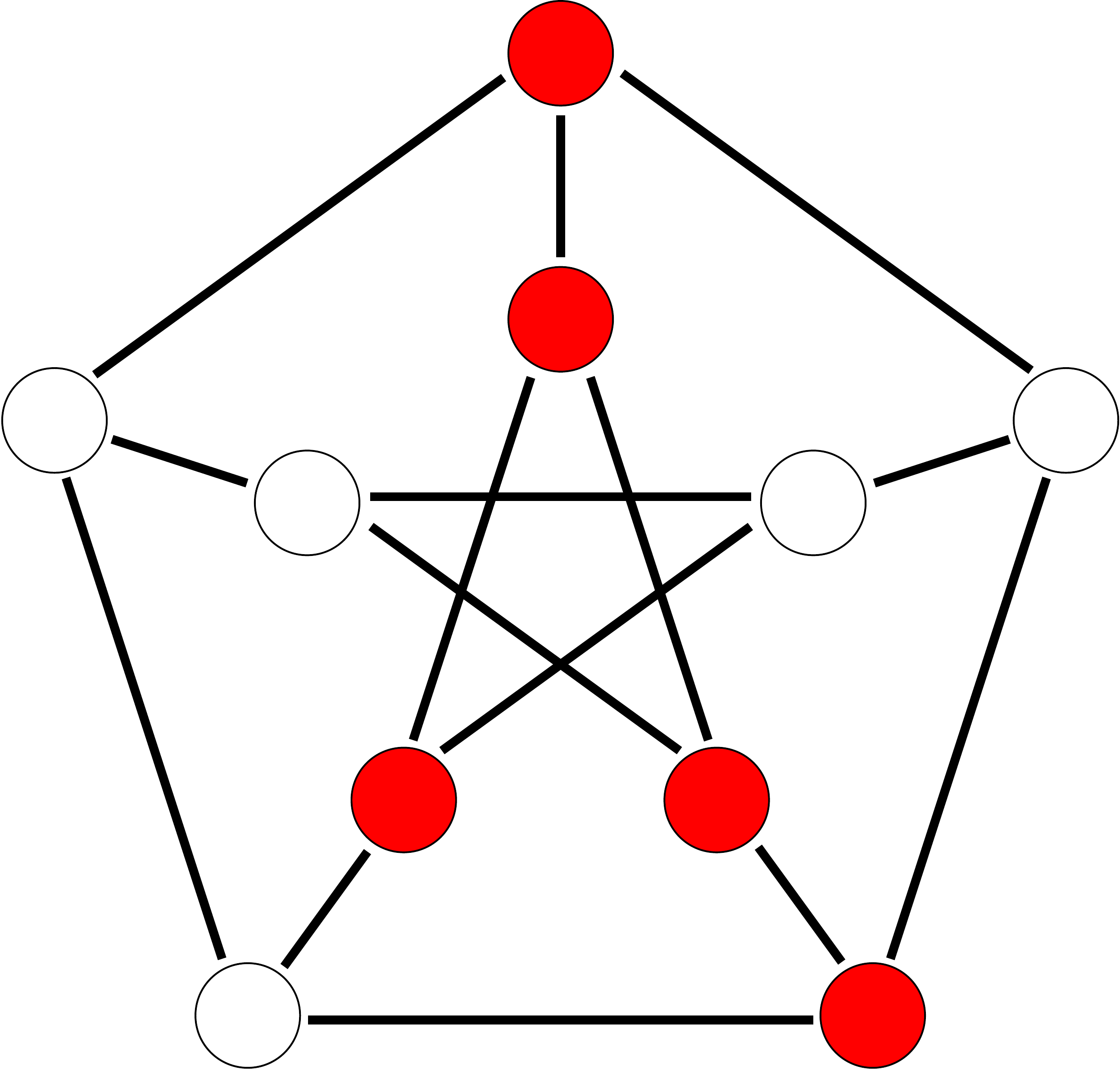}
    \end{minipage}
    \begin{minipage}[t]{0.19\textwidth}
        \includegraphics[width=\textwidth]{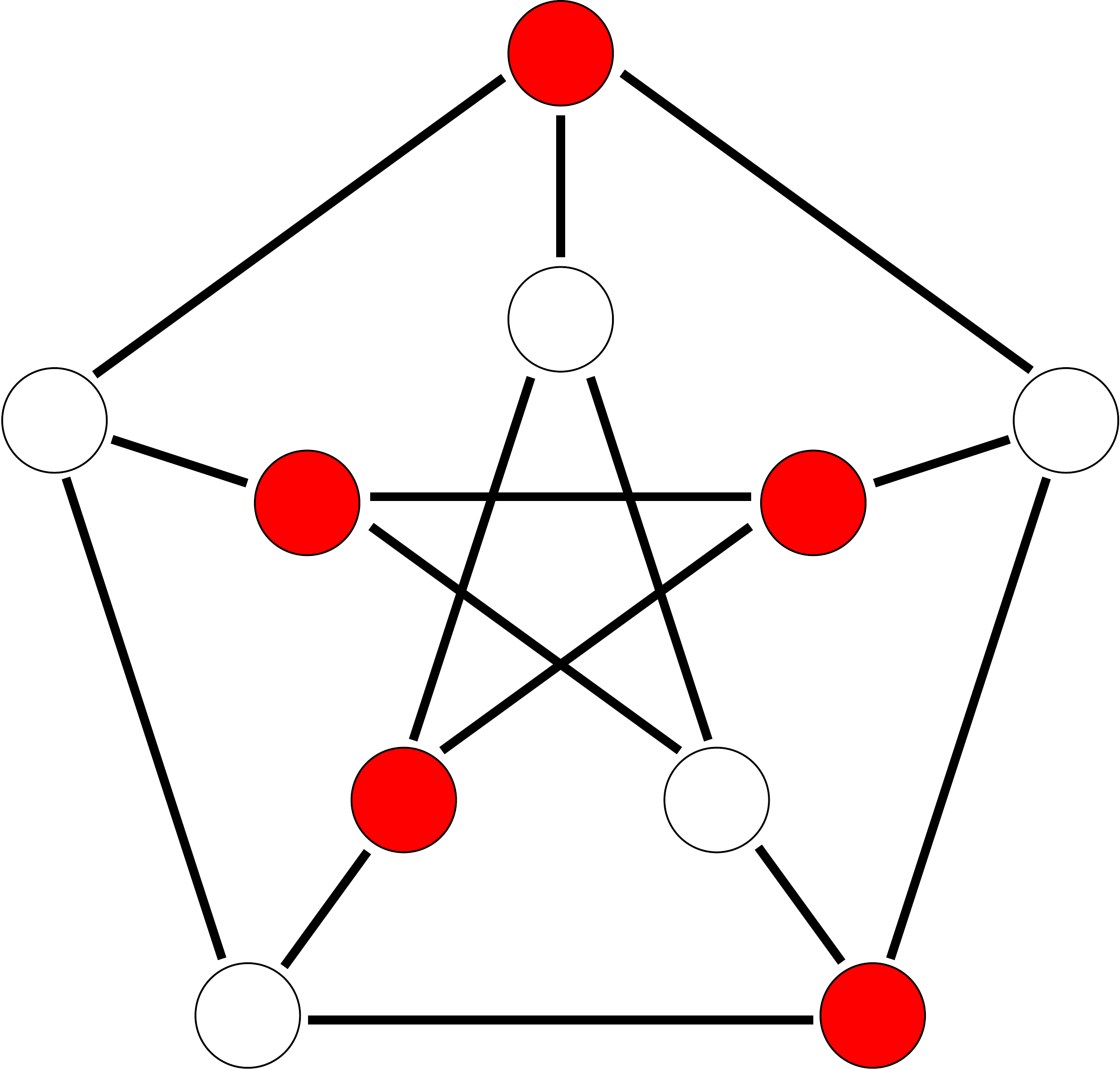}
    \end{minipage}
    \begin{minipage}[t]{0.19\textwidth}
        \includegraphics[width=\textwidth]{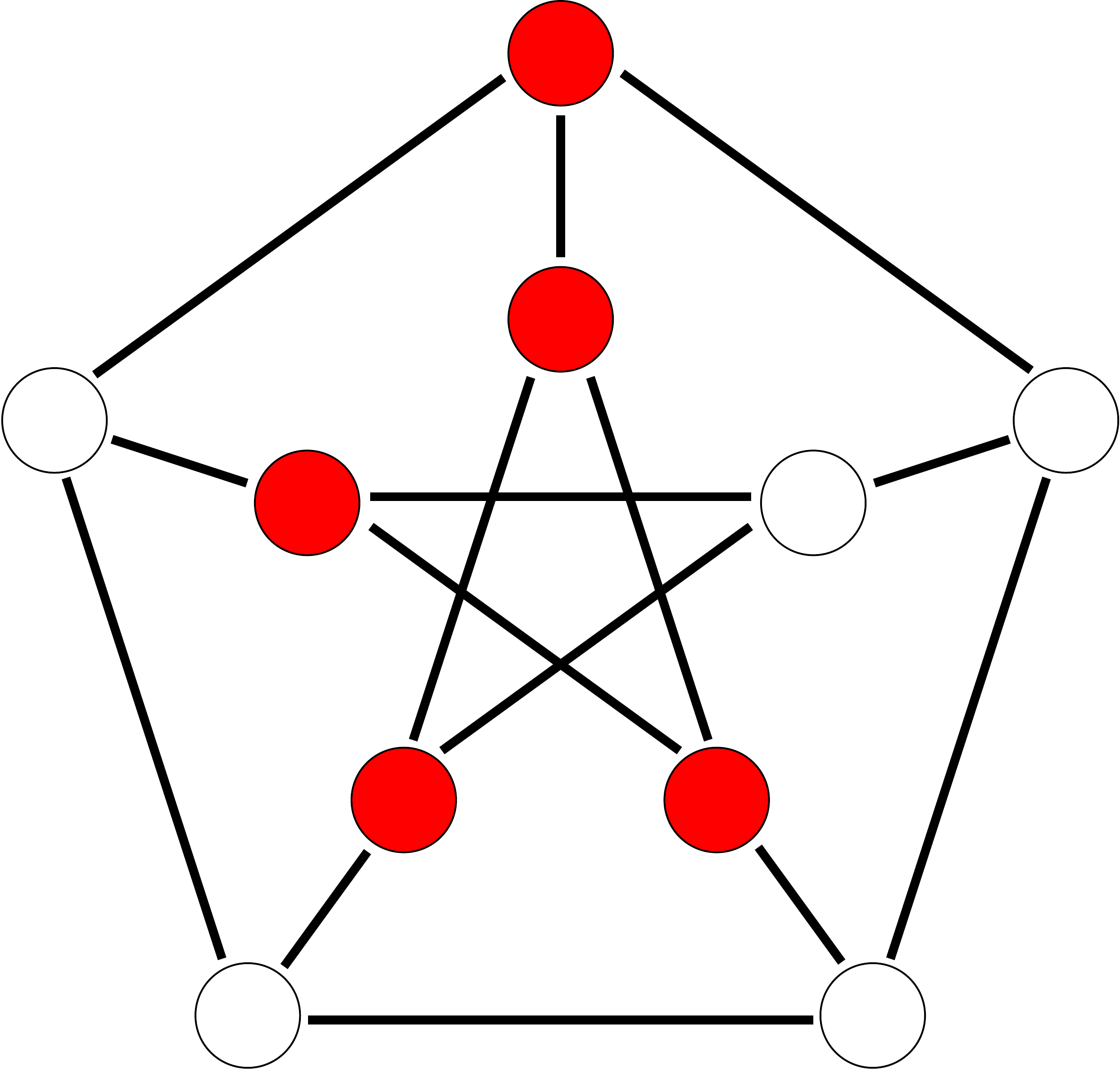}
    \end{minipage}
    \caption{Controllable networks (in green) and noncontrollable networks (in red) with five inputs.}
    \label{fig:Petersen_cnc}
\end{figure}

\end{example}

So far, we focused only on the controllability of a given system, but due to the duality of observability and controllability --~$ (A, C) $ is observable if and only if $ (A^T, C^T) $ is controllable~--, all the results described in this section can also be used to determine observability.

\section{Conclusion}
\label{sec:Conclusion}

In this paper, we analyzed the relationship between symmetries in complex networked systems and the controllability and observability of such systems using results from representation theory and illustrated that symmetries, depending on the symmetry type, might decrease the controllability. Furthermore, we derived necessary conditions on the number of inputs to guarantee controllability and presented an algorithm which computes a sparse input matrix $ B $ based on projections onto the isotypic components. We demonstrated the algorithm using two guiding examples, a network with $ D_4 $ symmetry and the Petersen graph, whose automorphism group is isomorphic to $ S_5 $.

Future work would be to analyze how the results presented within this paper can be extended to time-varying systems of the form
\begin{equation}
    \begin{split}
        \dot{x} &= A(t) \, x + B(t) \, u, \\
        y       &= C(t) \, x + D(t) \, u.
    \end{split}
\end{equation}
For such a system, controllability and observability can be defined in terms of the so-called \emph{Controllability Gramian} and \emph{Observability Gramian}. One possible extension would be to analyze \emph{structural} controllability and observability of these networks using the nonzero patterns of the matrices. The time-varying network structure could result in a temporarily uncontrollable system if a matrix $ A(t_0) $ without symmetries at time $ t_0 $ is changed into a highly symmetric matrix $ A(t_1) $ at time $ t_1 $. Our results show that for such a system the number of required control inputs might increase significantly. Consider for instance the guiding example: If $ C_1 \ne C_2 $, the system is controllable from one node, for $ C_1 = C_2 $, on the other hand, two control inputs are required. Equivalently, adding links to or removing links from a network might cause new symmetries or break existing symmetries and, as a result, change the number of required control inputs.

\bibliographystyle{unsrt}
\bibliography{SCSN}

\end{document}